\documentclass[a4paper]{article}

\usepackage[utf8]{inputenc}
\usepackage{lmodern}
\usepackage[T1]{fontenc}
\usepackage{microtype}
\usepackage{amsmath, amssymb, amsthm}
\usepackage{dsfont}

\numberwithin{equation}{section}

\newcommand{\1}{\mathds{1}}
\newcommand{\K}{\mathds{K}}
\newcommand{\R}{\mathds{R}}
\newcommand{\C}{\mathds{C}}
\newcommand{\N}{\mathds{N}}
\newcommand{\A}{\mathcal{A}}
\newcommand{\B}{\mathcal{B}}
\newcommand{\G}{\mathcal{G}}

\newcommand{\D}{\mathcal{D}}
\newcommand{\V}{\mathcal{V}}

\newcommand{\fra}{\mathfrak{a}}
\newcommand{\frb}{\mathfrak{b}}

\newcommand\norm[1]{\lVert #1 \rVert}
\newcommand\abs[1]{\lvert #1 \rvert}

\theoremstyle{plain}
\newtheorem{theorem}{Theorem}[section]
\newtheorem{proposition}[theorem]{Proposition}
\newtheorem{lemma}[theorem]{Lemma}
\newtheorem{corollary}[theorem]{Corollary}

\theoremstyle{remark}
\newtheorem{remark}[theorem]{Remark}

\allowdisplaybreaks
\title{Non-Autonomous Maximal Regularity in Hilbert Spaces}
\author{Dominik Dier, Rico Zacher}

\begin{document}

\maketitle

\begin{abstract}\label{abstract}
\noindent We consider  non-autonomous evolutionary problems of the form
\[
u' (t)+A(t)u(t)=f(t), \quad u(0)=u_0,
\]
on $L^2([0,T];H)$, where $H$ is a Hilbert space.
We do not assume that the domain of the operator $A(t)$ is constant in time $t$, 
but that $A(t)$ is associated with a sesquilinear form $\fra(t)$.
Under sufficient time regularity of the forms $\fra(t)$ we prove well-posedness with maximal regularity in $L^2([0,T];H)$.
Our regularity assumption is significantly weaker than those from previous results 
inasmuch as we only require a fractional Sobolev regularity with arbitrary small Sobolev index.
\end{abstract}

\medskip


\section{Introduction}

Let $\K$ be the field $\R$ or $\C$ and let $V$ and $H$ be Hilbert spaces over the field $\K$ such that $V \overset d \hookrightarrow H$; i.e., $V$ is continuously and densely embedded in $H$. 
Then $H \overset d \hookrightarrow V'$ via $v\mapsto (v \, \vert \, \cdot )_H$, where $V'$ denotes the antidual (or dual if $\K=\R$) of $V$.
Let $I:=[0,T]$ where $T>0$.
Suppose $\fra \colon I \times V \times V \to \K$ is a bounded quasi-coercive non-autonomous form;
i.e., $\fra(t, \cdot, \cdot)$ is sesquilinear for all $t \in I$, $\fra(\cdot,v,w)$ is measurable for all $v,w \in V$,
and there exist constants $M, \eta >0$ such that
\begin{align*}
	&\abs{\fra(t,v,w)} \le M \norm{v}_V \norm{w}_V  &(t \in I,\ v,w \in V),\\
	&\operatorname{Re} \fra(t,v,v) \ge \eta \norm{v}_V^2 - M \norm{v}_H^2 &(t \in I,\ v \in V).
\end{align*}
We define the operator $\A(t) \in \mathcal{L}(V,V')$ by $\A(t)v := \fra(t,v,\cdot)$ and the operator $A(t) \colon D(A(t)) \to H$ by $D(A(t)) := \{v \in V: \A(t)v \in H \}$, $A(t)v := \A(t)v$
for all $t \in I$.

A famous result due to J.~L.~Lions (see \cite[p.\ 513]{DL92}) states that the non-autonomous Cauchy problem
\begin{equation}\label{eq:CP}
	u'+ \A(\cdot) u(\cdot) = f, \quad u(0) = u_0
\end{equation}
is \emph{well-posed with maximal regularity in $V'$} and the \emph{trace space} is $H$;
i.e., for every $f \in L^2(I;V')$ and $u_0\in H$ there exists a unique $u \in H^1(I;V') \cap L^2(I;V) \hookrightarrow C(I;H)$ that solves \eqref{eq:CP}.

We say that $\fra$ has $H$ \emph{maximal regularity}  if 
for all $f \in L^2(I;H)$ and $u_0=0$ the solution $u$ of \eqref{eq:CP} is in $H^1(I;H)$,
and consequently in
\[
	\textit{MR} := \{ u \in L^2(I;V) \cap H^1(I;H) : \A u \in L^2(I;H) \}.
\]
It is easy to see that if $\fra$ has $H$ maximal regularity, then the solution $u$ of \eqref{eq:CP} is in $H^1(I;H)$ for every $f \in L^2(I;H)$ and $u_0 \in \textit{Tr}$,
where the trace space $\textit{Tr}$ is defined by $\textit{Tr} = \{v(0) : v \in \textit{MR} \}$.

The problem of non-autonomous $H$ maximal regularity has been studied extensively in the literature.
In the autonomous case, i.e.\ $\fra(\cdot,v,w)$ is constant for every $v,w \in V$, additional regularity of the inhomogeneity $f$ and the initial value $u_0$ leads to higher regularity of the solution $u$.
In particular, it is known that one has maximal regularity in $H$ with $\textit{Tr} = D(A(0)^{1/2})$.
As shown recently in \cite[p.~36]{Die14}, the property of $H$ maximal regularity fails in general in the non-autonomous case, that is without further regularity assumptions on the form $\fra$.
If the form is additionally \emph{symmetric}, i.e.\ $\fra(t,v,w) = \overline{\fra(t,w,v)}$ for all $t \in I$, $v,w \in V$,
the problem of $H$ maximal regularity was explicitly asked by Lions and is still open (see \cite[p.\ 68]{Lio61}).

Lions himself proved $H$ maximal regularity if $\fra$ is symmetric and $\A(\cdot) \in C^1(I;\mathcal{L}(V;V'))$ (see \cite[p.\ 65]{Lio61}).
Using a different approach, $H$ maximal regularity was established in \cite{OS10}, assuming that $\A(\cdot) \in C^{\alpha}(I; \mathcal{L}(V,V'))$ for some $\alpha > 1/2$, without symmetry assumption.
This result was further improved in \cite{HO14}, where the aforementioned Hölder condition is replaced by a weaker ``Dini'' condition for $\fra$, which can be viewed as a generalization of the Hölder condition above to the limiting case $\alpha = 1/2$.
Moreover, they established $L^p(I;H)$ maximal regularity for $1<p<\infty$.

Lions' result was recently generalized in another direction in \cite{Die15}.
Assume in addition that $\fra$ is symmetric and of \emph{bounded variation}; i.e., there exists a bounded and non-decreasing function $g \colon I \to \R$ such that
\begin{equation*}
	\abs{\fra(t,v,w)- \fra(s,v,w)} \le [g(t)-g(s)] \norm{v}_V \norm{w}_V \quad (0\le s \le t \le T,\ v,w \in V).
\end{equation*}
Then $\fra$ has maximal regularity in $H$ with $\textit{Tr}= V$, and $\textit{MR}$ is continuously embedded in $C(I;V)$.

More recent further contributions to maximal regularity for non-autonomous problems are \cite{ADO14}, \cite{ADLO14}, \cite{ACFP07}, \cite{SP01}, \cite{Ama04}.

The main contribution of the present article is a general result on higher regularity of solutions to the non-autonomous problem \eqref{eq:CP}, see Theorem~\ref{thm:MRinRforformsonI} below.
As a special case it contains the following result on $H$ maximal regularity.
\begin{corollary}\label{cor:intro}
	Suppose that in addition $\A(\cdot)$ belongs to the homogeneous fractional Sobolev space $\mathring W^{1/2+\delta, 2}(I; \mathcal{L}(V,V'))$ for some $\delta>0$; i.e.,
	\[
		\int_I\int_I \frac{\norm{\A(t)-\A(s)}_{\mathcal{L}(V,V')}^2}{\abs{t-s}^{2+2\delta}}  \  \mathrm{d}{t}\ \mathrm{d}{s} < \infty.
	\]
	Then \eqref{eq:CP} has $H$ maximal regularity with $\textit{Tr} = D(A(0)^{1/2})$. Moreover, $\textit{MR}$ embeds continuously in $H^{1/2}(I;V)$.
\end{corollary}
Note that $W^{1/2+\delta, 2}(I; \mathcal{L}(V,V')) \hookrightarrow C(I; \mathcal{L}(V,V'))$ and we identify $\A(\cdot)$ with its continuous version.

This result closes the gap between the Hölder and the bounded variation assumption on the form $\fra$ in the following sense:
It holds $C^{1/2+\delta}(I;\mathcal{L}(V;V')) \hookrightarrow \mathring W^{1/2+\delta, 2}(I; \mathcal{L}(V,V'))$.
Moreover, without symmetry of the form, by the counterexample mentioned above, bounded variation does not suffice for $H$ maximal regularity.
However, if we replace this assumption by the slightly stronger assumption $\A \in \mathring W^{1,1+\delta}(I;\mathcal{L}(V,V'))$, for some $\delta >0$,
we also obtain that $\mathring W^{1,1+\delta}(I;\mathcal{L}(V,V')) \hookrightarrow \mathring W^{1/2+\delta', 2}(I; \mathcal{L}(V,V'))$ with $\delta' = \frac\delta{1+\delta}$.
Thus Corollary~\ref{cor:intro} applies and yields $H$ maximal regularity.

We like to point out, that Theorem~\ref{thm:MRinRforformsonI} does not only treat the case of $H$ maximal regularity, but
covers the whole range of complex interpolation spaces $[H,V']_{1-2\alpha}$, where $\alpha \in (0,1/2]$.
Maximal regularity with respect to this space will be obtained under the additional assumption that $\A(\cdot)$ is in $\mathring W^{\alpha+\delta, 1/\alpha}(I; \mathcal{L}(V,V'))$ for some $\delta>0$.

Further, we investigate perturbations of lower order as in \cite{AM14}, \cite{Ouh14}.
As an application they treat non-autonomous Robin boundary conditions.
Again we significantly relax the regularity assumption from a Hölder condition to a fractional Sobolev space condition.

Our approach relies on elementary Hilbert space methods, such as the Lax--Milgram lemma and Plancherel's theorem.
A key idea in the proof of Corollary~\ref{cor:intro} is to test equation \eqref{eq:CP} 
not only with $u$ and $u'$ but also with $\mathcal H u'$, where $\mathcal H$ denotes the Hilbert transform.
This is crucial to obtain a bound for the $H^{1/2}(\R;V)$ norm of $u$.

The present article is organized as follows.
Section~2 is of preliminary character. We provide some well known results about fractional powers of operators associated with forms and complex interpolation spaces.
Section~3 is concerned with abstract maximal regularity results on $I=\R$. Here we discuss conditions on operators 
$\A \in \mathcal{L}(L^2(\R;V), L^2(\R;V'))$ such that the Cauchy problem of the form $u'+\A u =f$ is well posed with maximal regularity in $V'$, $H$ and in the spaces `in between'.
Section~4 is devoted to non-autonomous forms and their associated operators.
In Section~5 we apply our abstract maximal regularity result of Section~3 to non-autonomous forms on $I=\R$, and in 
Section~6 we treat initial value problems by reducing them to the situation of Section~5.
In Section~7 we illustrate our results from Section~6 with applications to parabolic problems in divergence form (scalar equations and systems)
and to problems related to generalized fractional Laplacians.
Finally, in the appendix we collect some facts about Banach space valued fractional Sobolev spaces on the real line.
\section{Interpolation of the Gelfand triple}
Let $V$ and $H$ be Hilbert spaces over the field $\K$,
such that $V \overset d \hookrightarrow H$; i.e., $V$ is continuously and densely embedded in $H$. 
Then there exists a constant $c_H$ such that
\begin{equation}\label{eq:VinH}
	\norm{v}_H \le c_H \norm{v}_V \quad (v \in V).
\end{equation}
We denote by $V'$ the antidual (or dual if $\K=\R$) of $V$. Furthermore, we embed $H$ into $V'$ by the mapping
\begin{equation}\label{eq:embedding}
	j\colon v\mapsto (v \, \vert \, \cdot )_H.
\end{equation}
Then $(u \, \vert \, v)_H = \langle j(u), v \rangle$ for all $u \in H$ and $v \in V,$ where $\langle \cdot , \cdot \rangle$ denotes the duality pairing between $V'$ and $V$.
Moreover, $H$ is dense in $V'$ and 
\[
	\norm{j(u)}_{V'} \le c_H \norm u_H \quad (u \in H),
\]
where $c_H$ is the same constant as in \eqref{eq:VinH}.
It is convenient to identify $V$ and $H$ as subspaces of $V'$.
This means with respect to \eqref{eq:embedding} that we identify $u \in H$ with $j(u) \in V'$.

We define $\B \in \mathcal{L}(V,V')$ by $\B v = (v\, \vert \, \cdot)_V$.
Note that this is the associated operator of the scalar product in $V$.
Since this is a symmetric and coercive sesquilinear form $\B$ is invertible, defines a sectorial operator on $V'$ and
the part $B$ in $H$ of $\B$ defines a self adjoint operator (see \cite[p.~15]{Ouh05}).

We define the Hilbert space $H_\gamma$, where $\gamma \in [-1,1]$ by
\[
	H_\gamma:= \{ v \in D(\mathcal B^{\gamma/2}) : \B^{\gamma/2} v \in H \},\quad \norm{v}_{H_\gamma} := \norm{\B^{\gamma/2} v}_H,
\]
where $\B^{\gamma/2}$ denotes the fractional power of the sectorial operator $\B$ (see e.g.\ \cite[p.\ 163]{ABHN11}).

\begin{proposition}\label{prop:complexinterpolation}
	If $\gamma \in [0,1]$, then $H_\gamma = [H,V]_\gamma$ and $H_{-\gamma} = [H, V']_\gamma$,
	where we denote by $[\cdot, \cdot]_\gamma$ the complex interpolation space of order $\gamma$.
	In particular $H_{-1}=V'$, $H_0 =H$ and $H_1=V$.
\end{proposition}
\begin{proof}
	Let $v\in D(B)$. Since $B$ is self adjoint, we have $\norm{B^{1/2}v}^2_H = (Bv \, \vert \, v)_H = \norm{v}_V^2$.
	Thus $D(B^{1/2}) =V$ and since $\B^{-1/2} = B^{1/2} \B^{-1}$ we also have  $D(\B^{1/2})=H$.
	Since $B$ is self adjoint it has bounded imaginary powers.
	It follows that $\B= \B^{1/2} B B^{-1/2}$ has also bounded imaginary powers.
	Thus the claim follows by \cite[Theorem~4.17]{Lun09}.
\end{proof}

Let $\fra \colon V \times V \to \K$ be a sesquilinear form.
Moreover, we assume that $\fra$ is \emph{bounded}, i.e.\ there exists some $M\ge 0$ such that
\begin{equation*}
	\abs{\fra(v,w)} \le M \norm{v}_V \norm{w}_V \quad (v,w \in V)
\end{equation*}
and \emph{coercive}, i.e. there exists $\eta>0$ such that
\begin{equation*}
	\operatorname{Re} \fra(v,v) \ge \eta \norm{v}_V^2 \quad (v \in V).
\end{equation*}
Let $\A \in \mathcal{L}(V,V')$, $\A v= \fra(v, \cdot)$ be its \emph{associated operator} and $A \in \mathcal{L}(D(A),H)$, $D(A) = \{v \in V : \A v \in H\}$ its part in $H$.
Note that $\A$ defines a sectorial operator on $V'$.
\begin{proposition}\label{prop:trace}
	Let $\alpha \in (-1/2, 1/2)$.
	Then there exist constants $c,C>0$ depending only on $\alpha$, $M$ and $\eta$ such that
	\begin{equation*}
		c \norm{v}_{H_{2\alpha}} \le \norm{\A^\alpha v}_H \le C \norm{v}_{H_{2\alpha}} \quad (v \in H_{2\alpha}).
	\end{equation*}
\end{proposition}
\begin{proof}
	For the case $\alpha \in [0, 1/2)$ see \cite[Theorem 3.1]{Kat61}.
	The case $\alpha \in (-1/2, 0)$ follows by a duality argument.
	We define the adjoint sesquilinear form $\fra^* \colon V \times V \to \K$ by $\fra^*(v,w) = \overline{\fra(w,v)}$.
	Then $\fra^*$ is bounded and coercive with the same constants $M$ and $\eta$.
	Moreover, the part in $H$ of the associated operator $\A^*$ of $\fra^*$ is the adjoint operator of $A$ and 
	\begin{equation}\label{eq:Katobeta}
		c \norm{v}_{H_{2\beta}} \le \norm{(\A^*)^\beta v}_H \le C \norm{v}_{H_{2\beta}} \quad (v \in H_{2\beta}).
	\end{equation}
	for $\beta \in [0, 1/2)$, by the first part of the proof.
	Let $v,w \in V$, then we obtain by \eqref{eq:Katobeta} with $\beta=-\alpha$
	\begin{equation*}
		(\A^{\alpha}v \, \vert \, w)_H = (v \, \vert \, (\A^*)^{\alpha} w)_H \le \norm{v}_{H_{-2\alpha}} \norm{(\A^*)^{\alpha} w}_{H_{2\alpha}} \le  \norm{v}_{H_{-2\alpha}} \tfrac 1 c \norm{w}_{H}
	\end{equation*}
	and
	\begin{equation*}
		\langle v , w \rangle_{H_{2\alpha}, H_{-2\alpha}} =( \A^{\alpha} v  \, \vert \, (\A^*)^{-\alpha} w )_H \le \norm{\A^{\alpha} v}_H C \norm{w}_{H_{-2\alpha}}.
	\end{equation*}
	Taking the supremum over all $w \in V$ with $\norm{w}_H \le 1$ in the first inequality and with $\norm{w}_{H_{-2\alpha}} \le 1$ in the second proves the claim.
\end{proof}
\section{Maximal regularity on $\R$}

Let $V$ and $H$ be Hilbert spaces over the field $\K$ such that $V \underset d \hookrightarrow H$.
Furthermore, let $\A \in \mathcal{L}(L^2(\R;V), L^2(\R;V'))$ and set $M:= \norm{\A}$.
Suppose there exists some $\eta>0$ such that
\begin{equation}\label{eq:Acoercive}
	\operatorname{Re} \int_\R \langle \A v, v \rangle \ \mathrm{d}{t} \ge \eta \norm{v}_{L^2_V}^2 \quad (v \in L^2(\R;V)).
\end{equation}
Note that we denote the norm of $L^2(\R;X)$ by $\norm\cdot _{L^2_X}$ for any Hilbert space $X$.
\begin{theorem}\label{thm:weaksolutions}
	For every $f \in L^2(\R;V')$ there exists a unique $u \in \textit{MR}_0(\A):= L^2(\R;V)\cap H^1(\R;V')$ such that
	\begin{equation}\label{eq:weaksolution}
		u'+ \A u = f
	\end{equation}
	in $L^2(\R;V')$. In addition $u \in H^{1/2}(\R;H)$.
\end{theorem}
\begin{proof}
	We define the Hilbert space $\V_0 :=H^{1/2}(\R;H)\cap L^2(\R;V)$ with norm $\norm{v}_{\V_0}^2 := \norm{\partial^{1/2}v}_{L^2_H}^2 + \norm{v}_{L^2_V}^2$.
	Furthermore, we define the bounded sesquilinear form $E \colon \V_0 \times \V_0 \to \K$ by
	\begin{equation*}
		E(v,w) := \int_\R (\partial^{1/2}v \, \vert \,  \partial^{(1/2)*} ( 1 - \delta \mathcal{H} ) w)_H \ \mathrm{d}{t}
			 + \int_\R \langle \A v, ( 1 - \delta \mathcal{H} ) w \rangle \ \mathrm{d}{t}
	\end{equation*}
	where $\delta:=\frac{\eta}{M+1}$ and $\mathcal H$ is the Hilbert transform, i.e. the operator with Fourier symbol $-i \operatorname{sign} \xi$.
	Note that $E$ is bounded.
	Moreover, $E$ is coercive, since for $v \in \V_0$ we obtain by the boundedness of $\A$, \eqref{eq:Acoercive} and Parseval's relation that
	\begin{align*}
		\operatorname{Re} E(v,v) &\ge \operatorname{Re} \int_\R \abs{\xi} (i\operatorname{sign}(\xi)+\delta) \norm{\hat v}_H^2 \ \mathrm{d}{\xi} +  (\eta- \delta M) \norm{v}_{L^2_V}^2 \\
			&= \delta \norm{\partial^{1/2}v}_{L^2_H}^2 + (\eta- \delta M) \norm{v}_{L^2_V}^2 = \delta \norm{v}_{\V_0}^2,
	\end{align*}
	where $\hat v(\xi) = \frac 1 {\sqrt{2\pi}} \int_\R e^{-i t \xi} v(t) \, \mathrm{d} t$ denotes the Fourier transform of $v$.
	
	Let $f \in L^2(\R;V')$ and define $F \in \V_0'$ by
	\[
		F(w) := \int_\R \langle f , ( 1 - \delta \mathcal{H} ) w \rangle \ \mathrm{d}{t}.
	\]
	Then by the Lax--Milgram Lemma, there exists a unique $u \in \V_0$ such that
	\begin{equation}\label{eq:LMidentity}
		E(u,w) = F(w) \quad (w \in \V_0).
	\end{equation}
	Since $0<\delta<1$ we obtain that $0<1-\delta \le \abs{\delta i \operatorname{sign} \xi +1}\le 1 +\delta$ for all $\xi \in \R$.
	Thus, by Plancherel's theorem $1- \delta \mathcal{H}$ defines an isomorphism on $H^1(\R;V)$.
	Now \eqref{eq:LMidentity} implies
	\[
		-\int_\R \langle u , v' \rangle \ \mathrm{d}{t} + \int_\R\langle \A u, v \rangle \ \mathrm{d}{t}= \int_\R \langle f, v \rangle \ \mathrm{d}{t} \quad (v \in H^1(\R;V)).
	\]
	Hence $u \in H^1(\R;V')$ and $u$ satisfies \eqref{eq:weaksolution} by density of $H^1(\R;V)$ in $L^2(\R;V)$.
	On the other hand, any solution of \eqref{eq:weaksolution} satisfies \eqref{eq:LMidentity}, since $\textit{MR}_0(\A) \hookrightarrow \V_0$.
	This embedding is a consequence of the estimate
	\[
		\int_\R \big\lVert \abs{\xi}^{1/2}\hat u(\xi) \big\rVert^2_H \ \mathrm{d}{\xi} = \int_\R \langle \abs{\xi}\hat u(\xi), \hat u(\xi) \rangle  \ \mathrm{d}{\xi}  \le \norm{\xi \hat u}_{L^2_{V'}} \norm{\hat u}_{L^2_V}
	\]
	and Plancherel's theorem. 
	Thus $u$ is unique.
\end{proof}
In order to model evolutionary problems we introduce the following `causality' condition.
\begin{proposition}\label{prop:evolution}
	Suppose the operator $\A$ is as above and commutes with the function $\1_{(-\infty,t)}$ for all $t \in \R$ in the sense that
	\[
		\1_{(-\infty,t)} \A v =  \A (\1_{(-\infty,t)} v) \quad (v \in L^2(\R;V),\ t \in \R).
	\]
	Then for any $t \in \R$ and $u \in \textit{MR}_0(\A)$ 
	we have $u(s) = 0$ for all $s \le t$ if and only if $f(s):=u'(s)+\A u(s) = 0$ for a.e. $s < t$.
	Here, we identify $u$ with its continuous version with values in $V'$.
\end{proposition}
\begin{proof}
	Let $t \in \R$ and $u \in \textit{MR}_0(\A)$.
	Note that we have $\norm{u(\cdot)}_H^2 \in W^{1,1}(\R)$ with $\left(\norm{u(\cdot)}_H^2\right)' = 2 \operatorname{Re} \langle u' , u \rangle$ (see \cite[Proposition~1.2]{Sho97}).
	
	First suppose that $u(s) = 0$ for all $s \le t$. Then we obtain that 
	\[
		f(s) = u'(s)+ \A u(s) = 0 + \1_{(-\infty,t)}(s) \A u(s) = 0 \quad (\text{a.e.\ } s < t).
	\]
	Now suppose that $f(s) = 0$  for a.e.\ $s < t$.  We have
	\begin{multline*}
		\eta  \int_{-\infty}^t \norm{u}_V^2 \ \mathrm{d}{s} \le \operatorname{Re} \int_{-\infty}^t \langle \A u, u \rangle \ \mathrm{d}{s} = \operatorname{Re} \int_{-\infty}^t \langle f- u', u \rangle \ \mathrm{d}{s}\\
			 = - \frac 1 2 \int_{-\infty}^t (\norm{u}_H^2)'  \ \mathrm{d}{s} = - \frac 1 2 \norm{u(t)}_H^2.
	\end{multline*}
	Thus $u(s)=0$ for all $s \le t$.
\end{proof}
   
Next we prove higher regularity under stronger conditions on the operator $\A$, where we write $\A$ as the sum of a regular part $\A_1$ and a perturbation $\A_2$. 
Let $\alpha \in (0,1/2]$ and let $\A_1, \A_2 \colon L^2(\R;V) \to L^2(\R;V')$ be linear operators.
Suppose there exist constants $\eta, \eta_1 >0$ and $M \ge 0$ such that
\begin{align}
	\norm{\A_1 v}_{L^2_{V'}} \le M \norm{v}_{L^2_V}& \quad &(v \in L^2(\R;V)), \label{eq:A1bounded}\\
	\A_1 v \in H^\alpha(\R;V') \text{ \& } \norm{\partial^\alpha \A_1 v}_{L^2_{V'}} \le M \norm{v}_{H^\alpha_V}& \quad &(v \in H^\alpha(\R;V)), \label{eq:A1bounded2}\\
	\operatorname{Re}  \int_\R \langle \A_1 v, v \rangle \ \mathrm{d}{t} \ge \eta \norm{v}_{L^2_V}^2& \quad &(v \in L^2(\R;V)), \label{eq:A1coercive} \\ 
	\operatorname{Re} \int_\R \langle \partial^\alpha\A_1 v, \partial^\alpha v \rangle \ \mathrm{d}{t} \ge \eta_1 \norm{\partial^\alpha v}_{L^2_V}^2 -M \norm{v}_{L^2_V}^2& \quad &(v \in H^\alpha(\R;V)).\label{eq:alphaA1coercive}
\end{align}
Moreover, suppose that there exists $M_2 \ge 0$ and $\eta_2 < \eta$ such that
\begin{align}
		\norm{\A_2 v}_{L^2_{V'}} \le M_2 \norm{v}_{L^2_V}& \quad &(v \in L^2(\R;V)), \label{eq:A2bounded}\\
	\operatorname{Re} \int_\R \langle \A_2 v, v \rangle \ \mathrm{d}{t} \ge - \eta_2 \norm{v}_{L^2_V}^2 & \quad &(v \in L^2(\R;V)), \label{eq:A_2}
\end{align}
and that for every $\varepsilon >0$ there exists a constant $c_\varepsilon$ with
\begin{multline}\label{eq:A_22}
	\left\lvert \int_\R \langle \A_2 v,  \partial^\alpha \partial^{\alpha*} w \rangle  \ \mathrm{d}{t} \right\rvert \le \left[\varepsilon \norm{\partial^\alpha v}_{L^2_V} + c_\varepsilon \norm{v}_{L^2_V}\right] \\
		\times \left[ \norm{w}_{H^\alpha_V} + \norm{\partial^{2\alpha} w}_{L^2_{H_{1-2\alpha}}} \right] \quad
			(v \in H^{\alpha}(\R;V),\, w \in H^{2\alpha}(\R;V) ).
\end{multline}

Note that the operator $\A:=\A_1+\A_2$ is in $\mathcal{L}(L^2(\R;V), L^2(\R;V'))$ by \eqref{eq:A1bounded} and \eqref{eq:A2bounded}  and satisfies \eqref{eq:Acoercive} by \eqref{eq:A1coercive} and \eqref{eq:A_2}.
Thus, we may apply Theorem~\ref{thm:weaksolutions} to the operator $\A$.

We define the Hilbert space $\V_\alpha := H^{1/2+\alpha}(\R;H) \cap H^{\alpha}(\R;V)$ with norm
\[
	\norm{v}_{\V_\alpha}^2 := \norm{\partial^{1/2+\alpha}v}_{L^2_H}^2 + \norm{v}_{H^\alpha_V}^2.
\]
Furthermore, we define the maximal regularity space 
\[
	\textit{MR}_{\alpha}(\A) := \{v \in H^1(\R;H_{2\alpha-1}) \cap L^2(\R;V) : \A v \in L^2(\R;H_{2\alpha-1})\}
\]
with norm
\[
	\norm{v}_{\textit{MR}_{\alpha}(\A)}^2 := \norm{v'}_{L^2_{H_{2\alpha-1}}}^2 + \norm{\A v}_{L^2_{H_{2\alpha-1}}}^2. 
\]
Note that $\textit{MR}_{\alpha}(\A)$ is a Hilbert space.

\begin{theorem}\label{thm:MRonR}
	Let $\alpha\in(0,\frac 1 2]$ and $\A:= \A_1+ \A_2$, where $\A_1, \A_2 \colon L^2(\R;V) \to L^2(\R;V')$ are linear operators satisfying \eqref{eq:A1bounded}--\eqref{eq:A_22}.
	Then, for every $f \in L^2(\R;H_{2\alpha-1})$, there exists a unique $u \in \textit{MR}_{\alpha}(\A)$ such that 
	\begin{equation}\label{eq:CPonR}
		u'+ \A u = f
	\end{equation}
	in $L^2(\R;H_{2\alpha-1})$. Moreover, $\textit{MR}_{\alpha}(\A) \hookrightarrow \V_\alpha$.
\end{theorem}
For the proof of the theorem we begin with two lemmas.
\begin{lemma}\label{lem:interpolation}
	Let $\alpha \in (0, \frac 1 2]$ and $\delta\in [0,1-2\alpha]$.
	Suppose $u \in H^{\frac{\delta+1}2 +\alpha}(\R; H_{-\delta}) \cap H^{\delta +\alpha}(\R; H_{1-2\delta})$.
	Then $u \in H^{\delta + 2\alpha}(\R;H_{1-2\delta-2\alpha})$ and
	\begin{equation}\label{eq:interpolation}
		\norm{\partial^{\delta +2\alpha} u}_{L^2(\R;H_{1-2\delta-2\alpha})} \le  \norm{u}_{H^{\frac{\delta+1}2 +\alpha}(\R; H_{-\delta})}^{\frac{2\alpha}{1-\delta}}
				\norm{u}_{H^{\delta +\alpha}(\R; H_{1-2\delta})}^{1-\frac{2\alpha}{1-\delta}}.
	\end{equation}	
\end{lemma}
\begin{proof}
		Note that 
		\[
			H^{\delta + 2\alpha}(\R;H_{1-2\delta-2\alpha}) = [H^{\frac{\delta+1}2 +\alpha}(\R; H_{-\delta}), H^{\delta +\alpha}(\R; H_{1-2\delta})]_\lambda,
		\]
		with $\lambda = \frac{2\alpha}{1-\delta} \in [0,1]$. Now the claim follows by Proposition~\ref{prop:complexinterpolation} and \cite[p.~53]{Lun09}.
\end{proof}
\begin{lemma}\label{lem:MR}
	Let $\alpha\in(0,\frac 1 2]$ and $\A:= \A_1+ \A_2$, where $\A_1, \A_2 \colon L^2(\R;V) \to L^2(\R;V')$ are linear operators satisfying \eqref{eq:A1bounded}-\eqref{eq:A_22}.
	If $f \in L^2(\R;H_{2\alpha-1})$ and $u \in \textit{MR}_0(\A) \cap \V_\alpha$ such that $u'+\A u =f$,
			then $u \in \textit{MR}_\alpha(\A)$. 
\end{lemma}
\begin{proof}
	Let $\delta \in [0,1-2\alpha]$ and suppose that $u \in H^{\delta+\alpha}(\R; H_{1-2\delta})$.
	We show that $u$ is in $H^{\frac{\delta+1}2 +\alpha}(\R; H_{-\delta})$.
	Let $\rho \colon \R \to [0,\infty)$ be a mollifier and define the function $\rho_n \colon \R \to [0,\infty)$ by $\rho_n(t):= n \rho(n t)$ for $n \in \N$.
	We set $g_n := g*\rho_n$ for any $n \in \N$ and $g \in L^2(\R;V')$. Moreover, we denote by $\abs\partial$ the operator with Fourier symbol $\abs\xi$.
	Since $u_n \in H^1(\R;V)$, we obtain
	\begin{align*}
		&\norm{\partial^{\frac{\delta+1}2 +\alpha} u_n}_{L^2_{H_{-\delta}}}^2 = \norm{\B^{-\delta/2} \abs\partial^{\frac{\delta+1}2 +\alpha} u_n}_{L^2_H}^2\\
			&= \int_\R(\mathcal{H} u_n' \, \vert \, \B^{-\delta} \abs\partial^{\delta +2\alpha} u_n)_{H} \ \mathrm{d}{t} 
			= \int_\R \langle \mathcal{H} (f_n-(\A u)_n) , \B^{-\delta} \abs{\partial}^{\delta +2\alpha} u_n \rangle  \ \mathrm{d}{t}\\
			&= \int_\R ( \B^{\alpha-1/2}\mathcal{H} f_n \, \vert \, \B^{1/2-\delta-\alpha} \abs\partial^{\delta +2\alpha} u_n)_{H}  \ \mathrm{d}{t}\\
				&\quad- \int_\R \langle \mathcal{H} \abs\partial^\alpha (\A_1 u)_n, \B^{-\delta} \abs\partial^{\delta +\alpha} u_n \rangle\ \mathrm{d}{t}
				- \int_\R \langle (\A_2 u)_n ,  \abs\partial^{2\alpha + \delta}  \mathcal{H} \B^{-\delta} u_n \rangle \ \mathrm{d}{t}\\
			&= : R_1 + R_2+ R_3.
	\end{align*}
	We have
	\begin{equation*}
		\abs{R_1} \le \norm{f}_{L^2(\R;H_{2\alpha-1})} \norm{\partial^{\delta +2\alpha} u_n}_{L^2(\R;H_{1-2\delta-2\alpha})}.
	\end{equation*}
	By \eqref{eq:A1bounded2} we have
	\begin{equation*}
		\abs{R_2} \le M \norm{u}_{H^\alpha(\R;V)} \norm{\partial^{\delta +\alpha} u}_{L^2(\R;H_{1-2\delta})}.
	\end{equation*}
	Moreover, by \eqref{eq:A_22} for every $\varepsilon >0$ there exists some constant $c_\varepsilon$ such that
	\begin{equation*}
		\abs{R_3} \le \big[\varepsilon \norm{\partial^\alpha u}_{L^2_V} + c_\varepsilon \norm{u}_{L^2_V}\big] \big[ \norm{u}_{H^{\delta +\alpha}(\R; H_{1-2\delta})} + \norm{u_n}_{H^{\delta + 2\alpha}(\R;H_{1-2\delta-2\alpha})} \big].
	\end{equation*}
	We apply Lemma~\ref{lem:interpolation} and Young's inequality for products
	and obtain that there exists some constant $c>0$ such that
	\[
		\norm{\partial^{\frac{\delta+1}2 +\alpha} u_n}_{L^2_{H_{-\delta}}} \le c \left[\norm{u}_{H^{\delta +\alpha}(\R; H_{1-2\delta})}  + \norm{u}_{\V_\alpha}+ \norm{f}_{L^2(\R;H_{2\alpha-1})} \right].
	\]
	By this inequality and since $u_n \to u$ in $L^2(\R;V)$, we obtain that every subsequence of $(u_n)$ converges weakly to $u$ in $H^{\frac{\delta+1}2 +\alpha}(\R; H_{-\delta})$.
	Hence $u$ belongs to $H^{\frac{\delta+1}2 +\alpha}(\R; H_{-\delta})$.
	
	If $\alpha \ge \frac 1 4$ we choose $\delta = 1- 2\alpha$ and obtain that $u \in H^1(\R; H_{2\alpha -1})$
	and consequently $u \in \textit{MR}_\alpha(\A)$.
	In the case $\alpha < \frac 1 4$ we have to iterate. 
	We consider the sequence $\delta_n = 1-2^{-n}$.
	If $u \in H^{1-2^{-n}+ \alpha}(\R;H_{2^{1-n}-1})$, which is the case for $n=1$, then we obtain by the consideration above that $u \in H^{1-2^{-(n+1)}+ \alpha}(\R;H_{2^{-n}-1})$, provided that $1-2^{-n}+\alpha \le 1$.
	Now if $n$ is the maximal integer satisfying this inequality we choose $\delta = 1- 2\alpha$ 
	and obtain that $u \in H^1(\R; H_{2\alpha -1})$.
\end{proof}

\begin{proof}[Proof of Theorem~\ref{thm:MRonR}]
	Let  $\alpha \in (0,1/2]$ and $f \in L^2(\R;H_{2\alpha-1})$.
	Note that $\V_\alpha \hookrightarrow H^{2\alpha}(\R;H_{1-2\alpha}) \cap H^{\alpha}(\R;V)$ by Lemma~\ref{lem:interpolation} with $\delta =0$.
	Thus by \eqref{eq:A_22} the sesquilinear form $(v,w) \mapsto \int_\R \langle \A_2 v,  \partial^\alpha \partial^{\alpha*} w \rangle  \, \mathrm{d}{t} $ extends continuously to
	a sesquilinear form from $\V_\alpha \times \V_\alpha$ to $\K$. 
	
	We define the bounded sesquilinear form $E \colon \V_\alpha \times \V_\alpha \to \K$ by
	\begin{multline*}
		E(v,w) := \int_\R (\partial^{1/2+\alpha}v \, \vert \,  \partial^{(1/2-\alpha)*} [\partial^{\alpha*} \partial^\alpha( 1 - \delta \mathcal{H} ) +\rho ] w)_H \ \mathrm{d}{t}\\
			 +  \int_\R \langle \partial^\alpha \A_1 v, \partial^\alpha( 1 - \delta \mathcal{H} ) w \rangle \ \mathrm{d}{t}
			+  \int_\R \langle \A_2 v, \partial^{\alpha*}\partial^{\alpha} ( 1 - \delta \mathcal{H} ) w \rangle \ \mathrm{d}{t}\\
			+ \rho  \int_\R \langle (\A_1+\A_2) v,  w \rangle \ \mathrm{d}{t},
	\end{multline*}
	where we choose $\delta, \rho >0$ appropriately.
	Furthermore, we define $F \in \V_\alpha'$ by
	\[
		F(w) := \int_\R (\B^{\alpha- 1/2} f \, \vert \, \B^{-\alpha+ 1/2}[\partial^{\alpha*} \partial^\alpha( 1 - \delta \mathcal{H} )+\rho] w )_H \ \mathrm{d}{t}.
	\]
	
	We show later that $E$ is coercive. If this is the case, then by the Lax--Milgram Lemma, there exists a unique $u \in \V_\alpha$ such that
	\[
		E(u,w) = F(w) \quad (w \in \V_\alpha).
	\]
	The operator $D\colon H^{1/2+\alpha}(\R;V) \to H^{1/2-\alpha}(\R;V)$, $v \mapsto [\partial^{\alpha*} \partial^\alpha( 1 - \delta \mathcal{H} )+\rho] v$ is invertible, 
	since it has the symbol $(1+\delta i\operatorname{sign}(\xi))\abs{\xi}^{2\alpha} + \rho$.
	Let $v \in H^{1/2-\alpha}(\R;V)$ and set $w= D^{-1}v \in \V_\alpha$.
	Now the identity $E(u,w)=F(w)$ implies, as in the proof of Theorem~\ref{thm:weaksolutions}, that $u \in \textit{MR}_0(\A)$ with $u'+\A u = f$ in $L^2(\R;V')$.
	We conclude by Lemma~\ref{lem:MR} that $u\in\textit{MR}_\alpha(\A)$. Moreover, $u$ is unique by Theorem~\ref{thm:weaksolutions}.
	
	We finish the proof of the theorem by establishing coercivity of $E$.
	Let $v \in \V_\alpha$, then 
	\begin{align*}
		&\operatorname{Re} E(v,v) \ge \operatorname{Re} \int_\R [\abs{\xi}^{1+2\alpha} (i\operatorname{sign}(\xi)+\delta) +\rho i \xi]\norm{\hat v}_H^2 \ \mathrm{d}{\xi} \\ 
				&\quad+  (\eta_1 - \delta M) \norm{\partial^\alpha v}_{L^2_V}^2 + (\rho (\eta- \eta_2) -\delta M- M) \norm{v}_{L^2_V}^2\\
			&\quad - \left[\varepsilon \norm{\partial^\alpha v}_{L^2_V} + c_\varepsilon \norm{v}_{L^2_V}\right] 
					\left[ \norm{( 1 - \delta \mathcal{H} )v}_{H^\alpha(\R;V)} + \norm{\partial^{2\alpha} ( 1 - \delta \mathcal{H} )v}_{L^2_{H_{1-2\alpha}}} \right]\\
			&\ge \delta \norm{\partial^{1/2+\alpha}v}_{L^2_H}^2 +  (\eta_1- \delta M) \norm{\partial^\alpha v}_{L^2_V}^2
				 + (\rho (\eta-\eta_2) -\delta M- M) \norm{v}_{L^2_V}^2\\
			&\quad- \left[\varepsilon \norm{\partial^\alpha v}_{L^2_V} + c_\varepsilon \norm{v}_{L^2_V}\right] 2\sqrt{\delta^2 + 1}\norm{v}_{\V_\alpha}
	\end{align*}		
	by \eqref{eq:A1coercive}, \eqref{eq:alphaA1coercive}, \eqref{eq:A_2}, \eqref{eq:A_22} and Lemma~\ref{lem:interpolation}.
	Hence, $E$ is coercive for sufficiently small $\delta$, $\varepsilon$ and sufficiently large $\rho$ by Young's inequality for products.
\end{proof}

\section{Non-autonomous forms}\label{sec:forms}
Let $V, W, H$ be Hilbert spaces over the field $\K$ with $V, W \overset d \hookrightarrow H$.
Let $I \subset \R$ be a closed interval.
The mapping
\[
	\fra \colon I \times V \times W \to \K
\]
is called a \emph{non-autonomous form} if $\fra(t,\cdot,\cdot)\colon V \times W \to \K$ is sesquilinear for all $t \in I$ and
$\fra(\cdot,v,w) \colon I \to \K$ is measurable for all $v \in V$ and $w \in W$.

We say the non-autonomous form $\fra$ is \emph{bounded} if there exists a constant $M$ such that
\begin{equation}\label{eq:Vbounded}
	\abs{\fra(t,v,w)} \le M \norm v_V \norm w_W \quad (t \in I,\ v \in V,\ w\in W).
\end{equation}

\begin{proposition}\label{prop:assocop}
	There exists a unique operator $\A \in \mathcal{L}(L^2(I;V), L^2(I;W'))$ such that $\fra(t,v(t),w(t)) = \langle (\A v)(t), w(t) \rangle$ for a.e.\ $t \in I$, for all $v\in L^2(I;V)$ and $w \in L^2(I;W)$.
\end{proposition}
\begin{lemma}\label{lem:Phi}
	The mapping $\Phi \colon L^2(I;W') \to (L^2(I;W))'$, $v \mapsto \int_I \langle v(t), . \rangle \ \mathrm{d} t$ is an isometric isomorphism.
\end{lemma}
For a proof of the lemma see \cite[p.\ 98]{DU77}.
\begin{proof}[Proof of Proposition~\ref{prop:assocop}.]
	We define the bounded form 
	\[
		\tilde \fra \colon L^2(I;V) \times L^2(I;W) \to \K, \quad \tilde \fra(v,w) = \int_I \fra(t,v,w) \ \mathrm{d}{t},
	\]
	and we define 
	\(
		\tilde \A \in \mathcal{L}(L^2(I;V), (L^2(I;W))')
	\)
	by $\tilde \A v = \tilde\fra(v, \cdot)$.
	We set $\A := \Phi^{-1} \circ \tilde \A$,
	then by the definition of $\Phi$ and $\tilde \A$ we have
	\begin{equation*}\label{eq:defassociatedop}
		\int_I \langle \A v, w \rangle \ \mathrm{d}{t} = \langle \tilde A v, w \rangle = \tilde\fra(v,w) = \int_I \fra(t,v,w) \ \mathrm{d}{t}
	\end{equation*}
	for all $v \in L^2(I;V)$ and all $w \in L^2(I;W)$. 
\end{proof}

Let $\fra \colon I \times V \times W \to \K$ be a bounded non-autonomous form.
Then we call the operator $\A \in \mathcal{L}(L^2(I;V), L^2(I;W'))$ from Proposition~\ref{prop:assocop} the \emph{associated operator} of $\fra$ and we write $\A \sim \fra$.
Moreover we denote by $\A(t) \in \mathcal{L}(V, W')$ the operator $v \mapsto \fra(t,v, \cdot)$.

In the case that $V=W$ we call $\fra$ \emph{quasi-coercive} if there exist $\eta >0$ and $\omega \in \R$ such that
\begin{equation}\label{eq:qcoercive}
	\operatorname{Re} \fra(t,v,v)+\omega \norm{v}_H^2 \ge \eta\norm{v}_V^2 \quad (t \in I,\ v \in V)
\end{equation}
and \emph{coercive} if there exists $\eta >0$ such that
\begin{equation}\label{eq:coercive}
	\operatorname{Re} \fra(t,v,v) \ge \eta\norm{v}_V^2 \quad (t \in I,\ v \in V).
\end{equation}
Note that $\A+\omega$ and $\A(t)+\omega$ are invertible if $\fra \colon I\times V \times V\to \K$ is a non-autonomous bounded quasi-coercive form and $\A \sim \fra$.

\section{Maximal regularity for non-autonomous operators associated with forms on $\R$}
Let $V$ and $H$ be Hilbert spaces over the fied $\K$ with $V \overset d \hookrightarrow H$.
Suppose $\fra \colon \R \times V \times V \to \K$ is a bounded coercive non-autonomous form, where $M\ge 0$ and $\eta>0$ are constants such that \eqref{eq:Vbounded} and \eqref{eq:coercive} hold.
Let $\A \in \mathcal{L}(L^2(\R;V), L^2(\R;V'))$ be the associated operator of $\fra$.
\begin{theorem}\label{thm:weakMRinRforforms}
	For every $f \in L^2(\R;V')$ there exists a unique $u \in \textit{MR}_0(\A)$ such that
	\begin{equation}
		u' + \A u = f.
	\end{equation}
\end{theorem}
\begin{proof}
	It is easy to check, that $\A$ satisfies the conditions of Theorem~\ref{thm:weaksolutions}.
\end{proof}

Next we consider higher regularity. Let $\alpha \in (0,1/2]$, $\beta \in [0,\alpha)$ and let $\fra_1 \colon \R \times V \times V \to \K,\ \fra_2 \colon \R \times V \times H_{1+2\beta-2\alpha}\to \K$ be bounded non-autonomous forms and $\A_1 \sim\fra_1,\ \A_2 \sim\fra_2$. Thus there exist constants $M, M_2$ such that
\begin{align}
	\abs{\fra_1(t, v, w )} &\le M \norm{v}_{V} \norm{w}_{V}  & (t\in\R ,\, v,w \in V),\label{eq:asA_11}\\
	\abs{\fra_2(t, v, w) } &\le M_2 \norm{v}_{V} \norm{w}_{H_{1+2\beta -2\alpha}}  & (t\in\R ,\, v \in V,\, w \in H_{1+2\beta -2\alpha}).\label{eq:asA_12}
\end{align}
By the definition of $H_{2\alpha-2\beta-1}$ we have that it is a subspace of $V'$ thus the mapping $v \mapsto \A v :=(\A_1+\A_2)v$ defines a bounded operator from $L^2(\R;V)$ to $L^2(\R;V')$.
Moreover, we suppose that there exist constants $\eta>0$ and $\eta_2 < \eta$ such that
\begin{align}
	\operatorname{Re}  \fra_1 (t, v, v ) &\ge \eta \norm{v}_{V}^2  &(t\in\R ,\, v \in V), \label{eq:asA_21}\\
	\operatorname{Re} \fra_2 (t, v, v ) &\ge - \eta_2 \norm{v}_{V}^2 &(t\in\R ,\, v \in V). \label{eq:asA_22}
\end{align}
Note that the form $\fra \colon \R \times V \times V \to \K$, $\fra = \fra_1+\fra_2$ satisfies the conditions of Theorem~\ref{thm:weakMRinRforforms}.

\begin{theorem}\label{thm:MRinRforforms}
	Suppose in addition that $\A_1(\cdot) \in \mathring W^{\alpha+\delta_0, \frac 1 {\alpha}}(\R;\mathcal{L}(V,V'))$ 
	and if $\beta >0$ that $\A_2(\cdot) \in \mathring W^{\beta+\delta_0, \frac 1 {\beta}}(\R;\mathcal{L}(V,H_{2\alpha-2\beta-1}))$ for some $\delta_0 >0$.
	Then for every $f \in L^2(\R;H_{2\alpha-1})$ there exists a unique $u \in \textit{MR}_\alpha(\A)$ such that
	\begin{equation}
		u' + \A u = f.
	\end{equation}
	Moreover, $\textit{MR}_\alpha(\A) \hookrightarrow H^{\alpha}(\R;V)$.
\end{theorem}

For the proof it will be crucial to control the commutator $\abs\partial^\alpha \A_1 v- \A_1 \abs\partial^\alpha v$.
This will be established by the following lemma.
\begin{lemma}\label{lem:comest}
Let $X, Y$ be Hilbert spaces, $\gamma \in (0, 1/2]$ and $\G\in L^\infty(\R;\mathcal{L}(X;Y)) \cap \mathring W^{\gamma+\delta_0, \frac 1 \gamma}(\R;\mathcal{L}(X;Y))$ for some $\delta_0>0$.
We claim that for every $\varepsilon>0$ there exists some constant $c_\varepsilon$ such that
\begin{multline}\label{eq:comest}
	\left(\int_\R \int_\R \frac{\norm{(\G(t)-\G(s))v(s)}_{Y}^2}{\abs{t-s}^{1+2\gamma}} \ \mathrm{d}{s} \ \mathrm{d}{t} \right)^{1/2}\\
		 \le \varepsilon \norm{\partial^\gamma v}_{L^2(\R;X)} + c_\varepsilon \norm{ v}_{L^2(\R;X)} \quad (v \in H^\gamma(\R;X)).
\end{multline}
Moreover, the mapping $u(\cdot) \mapsto \G(\cdot) u(\cdot)$ belongs to $\mathcal{L}(H^\gamma(\R;X), H^\gamma(\R;Y))$.
\end{lemma}
\begin{proof}
Let $0<h<1$ and $M:= \norm{\G(t)}_{L^\infty(\R;\mathcal{L}(X,Y))}$.
We obtain by Fubini's theorem that
\begin{multline*}
	\left(\int_\R \int_{\R\setminus (t-h,t+h)} \frac{\norm{(\G(t)-\G(s))v(s)}_{Y}^2}{\abs{t-s}^{1+2\gamma}} \ \mathrm{d}{s} \ \mathrm{d}{t} \right)^{1/2}\\
		\le 2M\left(\int_\R \int_{\R\setminus (t-h,t+h)} \frac{\norm{v(s)}_{X}^2}{\abs{t-s}^{1+2\gamma}} \ \mathrm{d}{s} \ \mathrm{d}{t} \right)^{1/2}\\
		= 2M\left(\int_\R \int_{\R\setminus (s-h,s+h)} \frac{1}{\abs{t-s}^{1+2\gamma}} \ \mathrm{d}{t} \ \norm{v(s)}_{X}^2\ \mathrm{d}{s} \right)^{1/2}
		= \frac{2M}{\sqrt\gamma h^\gamma} \norm{v}_{L^2(\R;X)}
\end{multline*}
for all $v \in L^2(\R;X)$.
Next choose $0<\delta<\delta_0$ and $p> \frac 1 \gamma$ such that $p(\gamma+\delta) \le \frac 1 \gamma(\gamma+\delta_0)$.
Let $q$ be such that $\frac 1 2 = \frac 1 p + \frac 1 q$. For $v \in H^\gamma(\R;X)$, we obtain by Hölder's inequality
\begin{multline*}
	\left(\int_\R \int_{t-h}^{t+h} \frac{\norm{(\G(t)-\G(s))v(s)}_Y^2}{\abs{t-s}^{1+2\gamma}} \ \mathrm{d}{s} \ \mathrm{d}{t} \right)^{1/2}\\
		\le \left(\int_\R \int_{t-h}^{t+h} \frac{\norm{\G(t)-\G(s)}_{\mathcal{L}(X,Y)}^p}{\abs{t-s}^{1+p(\gamma+\delta)}} \ \mathrm{d}{s} \ \mathrm{d}{t}\right)^{1/p}
			  \left(\int_\R  \int_{t-h}^{t+h} \frac{\norm{v(s)}_{X}^{q}}{\abs{t-s}^{1-\delta q}} \ \mathrm{d}{s} \ \mathrm{d}{t}\right)^{1/q}.
\end{multline*}
Using the uniform boundedness of $\G$ by the constant $M$ we obtain that the first term on the right hand side is bounded by
\[
	(2M)^{\frac{\gamma p-1}{\gamma p}} [\G]^{\frac 1{p\gamma}}_{W^{\gamma+\delta_0, \frac 1 \gamma}(\R; \mathcal{L}(X; Y) )}, 
\]
which is finite by our assumptions on the form.
By Fubini's theorem we obtain for the second term that
\begin{multline*}
	\left(\int_\R  \int_{t-h}^{t+h} \frac{\norm{v(s)}_{X}^{q}}{\abs{t-s}^{1-\delta q}} \ \mathrm{d}{s} \ \mathrm{d}{t}\right)^{1/q}
	 = \left(\int_\R  \int_{s-h}^{s+h} \abs{t-s}^{\delta q-1} \ \mathrm{d}{t}  \ \norm{v(s)}_{X}^{q} \ \mathrm{d}{s}\right)^{1/q}\\
	 =  h^{\delta }\left(\frac2{\delta q} \int_\R  \norm{v(s)}_{X}^{q}\ \mathrm{d}{s}\right)^{1/q}.
\end{multline*}
Since $p> \frac 1 \gamma$ we obtain that $q < \frac 2{1-2\gamma}$ if $\gamma < \frac 1 2$ and $q<\infty$ if $\gamma = \frac 1 2$.
Now the inequality \eqref{eq:comest} follows by Lemma~\ref{lem:compembedding}.

Let $v \in H^\gamma(\R;X)$, then by Proposition~\ref{prop:isometry}
\begin{align*}
	C_\gamma \norm{\partial^\gamma \G v}_{L^2(\R;Y)} 
		&= \left(\int_\R\int_\R \frac{\norm{\G v(t)- \G v(s)}^2_{Y}}{\abs{t-s}^{1+2\gamma}} \ \mathrm{d} s \ \mathrm{d}{t} \right)^{1/2}\\
		&\le \left(\int_\R\int_\R \frac{\norm{\G(t) ( v(t)-  v(s))}^2_{Y}}{\abs{t-s}^{1+2\gamma}} \ \mathrm{d} s \ \mathrm{d}{t} \right)^{1/2}\\
			&\quad + \left(\int_\R\int_\R \frac{\norm{(\G (t)- \G(s)) v(s)}^2_{Y}}{\abs{t-s}^{1+2\gamma}} \ \mathrm{d} s \ \mathrm{d}{t} \right)^{1/2}\\
		&\le (M C_\gamma+ 1 ) \norm{\partial^\gamma v}_{L^2(\R;X)} + c_1 \norm{v}_{L^2(\R;X)},
\end{align*}
where $c_1$ is the constant from \eqref{eq:comest}.
Thus the mapping $u(\cdot) \mapsto \G(\cdot) u(\cdot)$ belongs to $\mathcal{L}(H^\gamma(\R;X), H^\gamma(\R;Y))$.
\end{proof}

\begin{proof}[Proof of Theorem~\ref{thm:MRinRforforms}]
It is our goal to show that $\A_1$ and $\A_2$ satisfy the conditions of Theorem~\ref{thm:MRonR}.
First we consider the operator $\A_1$. Note that \eqref{eq:asA_11} implies \eqref{eq:A1bounded} and \eqref{eq:asA_12} implies \eqref{eq:A1coercive}.
By Lemma~\ref{lem:comest} we obtain that $\A_1$ satisfies \eqref{eq:A1bounded2}.
Next we show that $\A_1$ satisfies \eqref{eq:alphaA1coercive}.
Let $v \in H^\alpha(\R;V)$, then by Corollary~\ref{cor:isometry}
\begin{multline*}
	C_\alpha \operatorname{Re} \int_\R \langle \partial^\alpha \A_1 v(t),  \partial^\alpha v(t)\rangle  \ \mathrm{d}{t}\\
		= \operatorname{Re} \int_\R\int_\R \frac{\langle (\A_1 v(t)- \A_1 v(s)), v(t)-v(s) \rangle}{\abs{t-s}^{1+2\alpha}} \ \mathrm{d} s \ \mathrm{d}{t}\\
		=\operatorname{Re} \int_\R\int_\R \frac{\langle \A_1(t) ( v(t)- v(s)), v(t)-v(s) \rangle} {\abs{t-s}^{1+2\alpha}}\ \mathrm{d} s \ \mathrm{d}{t}\\ 
		\quad\qquad+ \operatorname{Re}\int_\R\int_\R \frac{\langle (\A_1(t)-\A_1(s)) v(s), v(t)-v(s) \rangle}{\abs{t-s}^{1+2\alpha}} \ \mathrm{d} s \ \mathrm{d}{t}\\
		\ge \eta C_\alpha\norm{\partial^\alpha v}_{L^2(\R;V)}^2 - \left[\varepsilon \norm{\partial^\alpha v}_{L^2(\R;V)} + c_\varepsilon \norm{v}_{L^2(\R;V)}\right]
			\sqrt{C_\alpha}\norm{\partial^\alpha v}_{L^2(\R;V)}.
\end{multline*}
Here we use again Lemma~\ref{lem:comest} with $\gamma=\alpha$ for some $\varepsilon>0$.
If we choose $\varepsilon < \eta \sqrt{C_\alpha}$, then the desired estimate \eqref{eq:alphaA1coercive} follows by Young's inequality.

Next we consider the operator $\A_2$. The assumptions \eqref{eq:A2bounded} and \eqref{eq:A_2} are satisfied by \eqref{eq:asA_12} and \eqref{eq:asA_22}.
In the case $\beta=0$ the assumption \eqref{eq:A_22} is satisfied without any further conditions on $\A_2$.

For $\beta>0$ we have by Lemma~\ref{lem:comest}
\begin{multline*}
	\left\lvert\int_\R \langle \A_2 v,  \partial^\alpha \partial^{\alpha*} w \rangle \ \mathrm{d}{t} \right\rvert
		= \left\lvert \int_\R ( \B^{\alpha-\beta- 1 / 2} \partial^\beta \A_2 v \, \vert \, \B^{1/2 -\alpha+\beta}  \partial^\alpha \partial^{(\alpha-\beta)*} w )_H \ \mathrm{d}{t} \right\rvert \\
		\le \norm{\partial^\beta \A_2 v}_{L^2_{H_{2\alpha-2\beta-1}}} \norm{\partial^\alpha \partial^{(\alpha-\beta)*} w}_{L^2_{H_{1-2\alpha+2\beta}}}\\
		\le C \norm{v}_{H^\beta(\R;V)} \norm{\partial^{2\alpha-\beta} w}_{L^2_{H_{1-2\alpha+2\beta}}} \\
	  \le \left[\varepsilon \norm{\partial^\alpha v}_{L^2_V} + c_\varepsilon \norm{v}_{L^2_V}\right] 
		\left[ \norm{w}_{H^\alpha(\R;V)} + \norm{\partial^{2\alpha} w}_{L^2(\R;H_{1-2\alpha})} \right].
\end{multline*}
for all $v \in H^{\alpha}(\R;V) ,\, w \in H^{2\alpha}(\R;V)$, for any $\varepsilon>0$ and sufficiently large $c_\varepsilon$.
Here we apply the complex interpolation inequality (see \cite[p.~53]{Lun09}) on the space
\[
	H^{2\alpha-\beta}(\R; H_{1-2\alpha+2\beta}) = [H^\alpha(\R;H_1), H^{2\alpha}(\R;H_{1-2\alpha})]_{\frac{\alpha-\beta}\alpha}
\]
and Young's inequality.
Thus $\A_1$ and $\A_2$ satisfy the conditions of Theorem~\ref{thm:MRonR}.
\end{proof}

\begin{remark}
	The theorem extends to more general perturbations than $\A_2$, e.g.
	sums of such operators, provided that condition \eqref{eq:asA_22} holds for the sum of these operators.
\end{remark}
\begin{remark}
	The statement of Theorem~\ref{thm:MRinRforforms} implies that the mapping 
	\[
		T\colon \textit{MR}_\alpha(\A) \to L^2(\R;H_{2\alpha-1}), \quad Tu =(\partial+\A)u
	\]
	defines an isomorphism.
	Thus $T$ and $T^{-1}$ are bounded operators by the closed graph theorem.
	On the other hand we may see by our proofs that these bounds depend only on the constants 
	appearing in the conditions of the theorem.
\end{remark}
\section{Initial value problems}

Let $V, H$ be Hilbert spaces over the field $\K$ with $V \overset d \hookrightarrow H$ and let $I =[0,T]$ where $T>0$.
Suppose $\fra \colon I \times V \times V \to \K$ is a bounded coercive non-autonomous form, 
where $M\ge 0$, $\eta>0$ and $\omega \ge 0$ are constants such that \eqref{eq:Vbounded} and \eqref{eq:qcoercive} hold.
Let $\A \in \mathcal{L}(L^2(I;V), L^2(I;V'))$ be the associated operator of $\fra$.
For $\alpha \in [0,1/2]$ we define the maximal regularity (Hilbert) space 
\[
	\textit{MR}_{\alpha}(I;\A) := \{v \in H^1(I;H_{2\alpha-1}) \cap L^2(I;V) : \A v \in L^2(I;H_{2\alpha-1})\}
\]
with norm
\[
	\norm{v}_{\textit{MR}_{\alpha}(I;\A)}^2 := \norm{v'}_{L^2(I;H_{2\alpha-1})}^2 + \norm{\A v}_{L^2(I;H_{2\alpha-1})}^2.
\]
Note that $\textit{MR}_0(I;\A) \hookrightarrow C(I; H)$ by \cite[p.~106, Proposition~1.2]{Sho97}.
\begin{theorem}\label{thm:weakMRinRforformsonI}
	For every $f \in L^2(I;V')$ and $u_0 \in H$ there exists a unique $u \in \textit{MR}_0(I;\A) = H^1(I;V')\cap L^2(I;V)$ with
	\begin{equation}\label{eq:weakIVP}
		u' + \A u = f,\quad  u(0) = u_0.
	\end{equation}
\end{theorem}
This result is well known, at least in the case that $H$ is separable (see \cite[p.\ 513]{DL92}). In order to illustrate our strategy of reducing the case of an interval to $\R$ and for the sake of completeness we provide a proof.
\begin{proof}[Proof of Theorem~\ref{thm:weakMRinRforformsonI}]
	First note that $u \in \textit{MR}_0(I;\A)$ is a solution of $u' + \A u = f$, $u(0) = u_0$ 
	if and only if $v(t):=e^{-\omega t}u(t) \in \textit{MR}_0(I;\A)$ is a solution of $v' + (\A+\omega) v = e^{-\omega \cdot}f$, $v(0) = u_0$.
	Thus we may assume that $\fra$ is coercive, i.e.\ $\omega = 0$.
	
	It is our goal to apply Theorem~\ref{thm:weakMRinRforforms}.
	We extend $\fra$ on the complement of $I$ by $\frac 1 T \int_0^T \fra(t,\cdot, \cdot) \, \mathrm{d}{t}$.
	We denote this extension again by $\fra$. Then $\fra \colon \R \times V \times V \to \K$ is a bounded coercive non-autonomous form.
	In particular \eqref{eq:Vbounded} and \eqref{eq:qcoercive} hold with the same constants $M\ge 0$ and $\eta>0$.
	
	In the case $u_0=0$ we extend $f$ by $0$ on the complement of $I$, 
	then the restriction to $I$ of the solution $u$ given by Theorem~\ref{thm:weakMRinRforforms} satisfies \eqref{eq:weakIVP} and is unique by
	Proposition~\ref{prop:evolution}.
	
	For the case $u_0\in H\setminus\{0\}$ note that $H=[V,V']_{1/2}=(V,V')_{1/2,2}$ since $V$ and $V'$ are Hilbert spaces.
	By the trace method for the real interpolation spaces (see \cite[Corollary~1.14]{Lun09}) there exists some $v \in \textit{MR}_0([0,\infty);\A)$ with $v(0) =u_0$.
	We set $w(t):=v(-t) \in \textit{MR}_0((-\infty, 0];\A)$ and we extend $f$ to $\R$ by $w'(t)+ (\A w)(t)$ on $(-\infty, 0)$ and by $0$ on $(T,\infty)$.
	Again the restriction to $I$ of the solution $u$ given by Theorem~\ref{thm:weakMRinRforforms} satisfies \eqref{eq:weakIVP} and is unique by
	Proposition~\ref{prop:evolution}.
\end{proof}

Next we consider higher regularity. Let $\alpha \in (0,1/2]$, $\beta \in [0,\alpha)$ and let $\fra_1 \colon I \times V \times V \to \K,\ \fra_2 \colon V \times H_{1+2\beta-2\alpha}\to \K$ be bounded non-autonomous forms and $\A_1 \sim\fra_1,\ \A_2 \sim\fra_2$. Thus there exist constants $M, M_2$ such that
\begin{align}
	\abs{\fra_1(t, v, w )} &\le M \norm{v}_{V} \norm{w}_{V}  & (t\in I,\, v,w \in V),\\
	\abs{\fra_2(t, v, w)} &\le M_2 \norm{v}_{V} \norm{w}_{H_{1+2\beta -2\alpha}}  & (t\in I,\, v \in V,\, w \in H_{1+2\beta -2\alpha}).\label{eq:IboundA2}
\end{align}
By the definition of $H_{1+2\beta-2\alpha}$ we have that $V$ is a subspace of $H_{1+2\beta-2\alpha}$, thus the mapping $v \mapsto \A v :=(\A_1+\A_2)v$ defines a bounded operator from $L^2(I;V)$ to $L^2(I;V')$.
Moreover, we suppose that $\fra_1$ is quasi-coercive; i.e., there exists constants $\eta>0$ and $\omega \in\R$ such that
\begin{align}
	\operatorname{Re}  \fra_1 (t, v, v ) + \omega \norm{v}_H  &\ge \eta \norm{v}_{V}^2  &(t\in I,\, v \in V). \label{eq:asA_21I}
\end{align}

\begin{theorem}\label{thm:MRinRforformsonI}
	In addition suppose that $\A_1(\cdot) \in \mathring W^{\alpha+\delta_0, \frac 1 {\alpha}}(I;\mathcal{L}(V,V'))$ 
	and if $\beta >0$ that $\A_2(\cdot) \in \mathring W^{\beta+\delta_0, \frac 1 {\beta}}(I;\mathcal{L}(V,H_{2\alpha-2\beta-1}))$ for some $\delta_0 >0$.
	For $t\in I$ we denote by $A(t)$ the part of $\A_1(t)$ in $H$ if $\beta = 0$ and the part of $\A(t)$ in $H$ if $\beta < 1/2$.
	Then for every $f \in L^2(I;H_{2\alpha-1})$ and $u_0 \in H_{2\alpha}$ for $\alpha <1/2$ and 
	$u_0 \in D(A(0)^{1/2})$ for $\alpha =  1 /2$,
	 there exists a unique $u \in \textit{MR}_\alpha(I;\A)$ with
	\begin{equation}\label{eq:alpha_sol}
		u' + \A u = f, \quad u(0) = u_0.
	\end{equation}
	Moreover, $\textit{MR}_\alpha(I;\A) \hookrightarrow H^{\alpha}(I;V)$, $\textit{MR}_\alpha(I;\A) \hookrightarrow C(I;H_{2\alpha})$ if $\alpha < \frac 1 2$ and
	$u(t) \in D(A(t)^{1/2})$ for every $u\in \textit{MR}_{1/ 2}(I;\A)$ and every $t \in I$.
\end{theorem}
\begin{proof}
	It is our goal to apply Theorem~\ref{thm:MRinRforforms}.
	By the same rescaling argument as in the proof of Theorem~\ref{thm:weakMRinRforformsonI} we may assume that $\fra_1$ is coercive, i.e.\ $\omega=0$.
	We even may replace $\A_2$ by $\A_2+\lambda$ for some $\lambda \ge 0$.
	This will be necessary to obtain \eqref{eq:asA_22} for some $\eta_2 < \eta$.
	
	Since $\A_1(\cdot)$ belongs to the space $\mathring W^{\alpha+\delta_0,\frac 1 \alpha}(I; \mathcal{L}(V,V'))$ we may extend $\A_1(\cdot)$ to 
	$\mathring W^{\alpha+\delta_0,\frac 1 \alpha}(\R; \mathcal{L}(V,V'))$ by applying the operators $\mathcal E_0^l$ and $\mathcal E_T^r$ defined in Proposition~\ref{prop:const_ext}.
	Moreover, we extend $\A_2(\cdot)$ to $\mathring W^{\beta+\delta_0,\frac 1 \beta}(\R; \mathcal{L}(V, H_{2\alpha-2\beta-1}))$ in the same way if $\beta >0$ and by $0$ if $\beta=0$.
	To apply Theorem~\ref{thm:MRinRforforms} it remains to show that \eqref{eq:asA_22} holds for some $\eta_2 < \eta$ if we choose $\lambda$ sufficiently large.
	Indeed by \eqref{eq:IboundA2} and Young's inequality we obtain for sufficiently large $\lambda$ that
	\begin{multline*}
		\operatorname{Re} \langle \A_2(t) v,v \rangle + \lambda \norm{v}_H^2 \ge - \abs{ \langle \A_2(t) v,v \rangle} + \lambda \norm{v}_H^2
			\ge-  M_2 \norm{v}_V \norm{v}_{H_{1+2\beta-2\alpha}}\\ + \lambda \norm{v}_H^2
				\ge-  M_2 \norm{v}_V \norm{v}_H^{2\alpha-2\beta}\norm{v}^{1+2\beta-2\alpha}_V + \lambda \norm{v}_H^2
					 \ge -\frac \eta 2\norm{v}_V^2
	\end{multline*}
	for all $v \in V$ and all $t \in \R$.
	
	We proceed as in the proof of Theorem~\ref{thm:weakMRinRforformsonI}.
	In the case $u_0=0$ we extend $f$ by $0$ on the complement of $I$, 
	then the restriction to $I$ of the solution $u$ given by Theorem~\ref{thm:MRinRforforms} satisfies \eqref{eq:weakIVP} and is unique by
	Proposition~\ref{prop:evolution}.
	
	Next we consider the case $u_0 \neq 0$.
	By the trace method for real interpolation (see \cite[Corollary~1.14]{Lun09}) there exists some 
	\(
		v \in \begin{cases} \textit{MR}_\alpha([0,\infty);\A(0))  \!\! &: \beta >0\\ \textit{MR}_\alpha([0,\infty);\A_1(0)) \!\! &: \beta =0\end{cases}
	\)
	with $v(0) =u_0$,
	where we use Proposition~\ref{prop:trace} in the case $\alpha < 1/2$. Note that in the case $\beta = 0$ 
	we have $\textit{MR}_\alpha(I;\A) = \textit{MR}_\alpha(I;\A_1)$ with equivalent Norms, since $\A_2 \in \mathcal{L}(L^2(I;V); L^2(I;H))$.
	We set $w(t):=v(-t) \in 
	\begin{cases} \textit{MR}_\alpha((-\infty, 0];\A(0))   &: \beta >0\\ \textit{MR}_\alpha((-\infty, 0];\A_1(0))  &: \beta =0\end{cases}$ 
	and we extend $f$ to $\R$ by $w'(t)+ \A w(t)$ on $(-\infty, 0)$ and by $0$ on $(T,\infty)$.
	Finally, the restriction to $I$ of the solution $u$ given by Theorem~\ref{thm:MRinRforforms} satisfies \eqref{eq:alpha_sol} and is unique by
	Proposition~\ref{prop:evolution}.
	
	Moreover, $\textit{MR}_\alpha(I;\A) \hookrightarrow C(I;H_{2\alpha})$ if $\alpha < \frac 1 2$ (see \cite[Theorem 3.6.]{Zac05})
	and in the case $\alpha = \frac 1 2$ we have $u\vert_{[T,\infty)} \in \begin{cases} \textit{MR}_\alpha([T, \infty);\A(T))   &: \beta >0\\ \textit{MR}_\alpha([T, \infty);\A_1(T))  &: \beta =0\end{cases}$, 
	hence $u(T) \in D(A(T)^{1/2})$ again by \cite[Corollary~1.14]{Lun09}.
\end{proof}

\section{Applications}
This section is devoted to some applications of the results given in the previous
sections. We give examples illustrating the theory without seeking for generality.
Throughout this section we consider the field $\K=\R$.
\subsection*{Elliptic operators with time dependent $L^\infty$ coefficients}
Let $\Omega \subset \R^d$ be a bounded domain, where $d \in \N$.
Suppose $a_{jk} \in L^\infty(I \times \Omega)$, $j,k\in \{1,\dots,d\}$ satisfy  
\begin{equation*}
	\sum_{j,k=1}^d a_{jk}(t,x) \xi_j \xi_k \ge \eta \abs{\xi}^2 \quad (t \in I,\, x \in \Omega,\, \xi \in \R^d).
\end{equation*}
for some  $\eta>0$.
By $H^1_0(\Omega)$ we mean the closure of the test functions $\D(\Omega)$ in $H^1(\Omega)$.
We denote for $\beta \in [0,1)$ by  $H^\beta_0(\Omega)$ the space $[L^2(\Omega), H_0^{1}(\Omega)]_\beta$ and by $H^{-\beta}(\Omega)$ the space $[L^2(\Omega), H^{-1}(\Omega)]_\beta$, where $H^{-1}(\Omega) := (H^1_0(\Omega))'$. Note that $H^{-\beta}(\Omega) = (H^\beta_0(\Omega))'$ (see \cite[p. 72]{Tri95}).

\begin{corollary}
	Let $I=[0,T]$, $\alpha \in [0,1/2]$ and if $\alpha>0$ suppose that in addition $a_{jk} \in \mathring W^{\alpha+\delta,\frac1{\alpha}}(I; L^\infty(\Omega))$ for some $\delta>0$.
	If $\alpha= 1/2$ we also assume that $\Omega$ has Lipschitz boundary.
	Then for every $f \in L^2(I ;H^{2\alpha -1}(\Omega))$, $u_0 \in H_0^{2\alpha}(\Omega)$ there exists a unique $u \in H^\alpha(I;H^1_0(\Omega)) \cap H^1(I;H^{2\alpha-1}(\Omega))$ such that
	\begin{equation*}
          	u'  -  \operatorname{div} ((a_{jk}) \nabla u ) =  f, \quad u(0) =u_0.
	\end{equation*}
\end{corollary}
\begin{proof}
	We define the non-autonomous form $\fra_1 \colon I \times H^1_0(\Omega) \times H^1_0(\Omega)\to \R$ by 
	\begin{equation*}
		\fra_1(t,v,w) = \int_\Omega \sum_{j,k=1}^d  a_{jk}(t,x) \partial_j v (t,x) \,{\partial_k w (t,x)} \ \mathrm{d}{x}.
	\end{equation*}
	Then $\fra_1$ satisfies the conditions of Theorem~\ref{thm:MRinRforformsonI}.
	Note for the case $\alpha = 1/2$ that $D(A(0)^{1/2})$ from Theorem~\ref{thm:MRinRforformsonI} coincides with the space $H^1_0(\Omega)$ 
	(see \cite{AT03}).
\end{proof}
Note that the domain of the part of $\A_1(t)$ in $H_0^{2\alpha -1}(\Omega)$ is time dependent for $\alpha >0$, where $\A_1 \sim \fra_1$.

\subsection*{Elliptic operators on $\R^d$ with mixed regularity}
Let $I=[0,T]$ with $T>0$, $d \in \N$, $0<\alpha_0<\alpha<\frac 1 2$ and $0<\beta_0<\beta< 1$ such that $2\alpha_0 + \beta_0 = 1$.
Moreover, let $a_{jk} \in L^\infty(I \times\R^d)$, $j,k\in \{1,\dots,d\}$ such that there exits some $\eta>0$ with
\begin{equation*}
	\sum_{j,k=1}^d a_{jk}(t,x) \xi_j \xi_k \ge \eta \abs{\xi}^2 \quad (t \in I,\, x \in \R^d,\, \xi \in \R^d).
\end{equation*}
\begin{corollary}
	Suppose in addition that $a_{jk} \in \mathring W^{\alpha,\frac1{\alpha_0}}(I; \mathring W^{\beta,\frac d{\beta_0}}(\R^d))$ for all $j,k \in \{1, \dots, d\}$.
	Then for every $f \in L^2(I; L^2(\R^d))$ and every $u_0 \in H^1(\R^d)$ there exists a unique $u \in H^1(I;L^2(\R^d)) \cap H^{1/2}(I; H^1(\R^d))$ such that
	\[
		u'- \operatorname{div}((a_{jk}) \nabla u) = f, \quad u(0)=u_0.
	\]
\end{corollary}
\begin{proof}
	We set $V:= H^{1+\beta_0}(\R^d)$, $H:= H^{\beta_0}(\R^d)$ and we identify $H$ with $H'$, then $V' = H^{\beta_0-1}(\R^d)$.
	Hence, $H_{2\alpha_0-1}=H_{-\beta_0} = [H, V']_{\beta_0} = L^2(\R^d)$.
	It is our goal to apply Theorem~\ref{thm:MRinRforformsonI} in this setting.
	We define $\frb \colon H^{\beta_0}(\R^d) \times H^{\beta_0}(\R^d) \to \R$ by
	\[
		\frb(v,w) = \int_{\R^d} \int_{\R^d} \frac{(v(x)-v(y)) {(w(x)-w(y))}}{\abs{x-y}^{2\beta_0}}\ \frac{\mathrm{d} x\ \mathrm{d} y}{\abs{x-y}^d} + (v \, \vert \, w)_{L^2(\R^d)}.
	\]
	Then $\frb$ is a scalar product on $H^{\beta_0}(\R^d)$. Let $\B \in \mathcal{L}( H^{\beta_0}(\R^d),( H^{\beta_0}(\R^d))')$ be its associated operator and $B \in \mathcal{L}(H^{2\beta_0}(\R^d), L^2(\R^d))$ the part of $\B$ in $L^2(\R^d)$.
	
	Moreover, we define the non-autonomous form
	\[
		\fra \colon I \times H^{1+\beta_0}(\R^d) \times H^{1+\beta_0}(\R^d) \to \R, \quad \fra(t,v,w) = \sum_{j,k=1}^d \frb \big(a_{jk}(t) \partial_k v, \partial_j w \big)
	\]
	and we denote its associated operator by $\A$.
	By our assumptions on the coefficients $a_{jk}$ we obtain that $\fra$ is bounded, quasi-coercive and that
	\[
		\A(\cdot) \in \mathring W^{\alpha,\frac1{\alpha_0}}(I; \mathcal{L}(H^{1+\beta_0}(\R^d), (H^{1+\beta_0}(\R^d))')).
	\]
	We only show that $\fra$ is quasi-coercive, the other properties are easy to check.
	Let $v \in H^{1+\beta_0}(\R^d)$, then
	\begin{align*}
		\fra(t,v,v)
		&=\sum_{j,k=1}^d \biggl[ \int_{\R^d} \int_{\R^d} a_{jk}(t,x)\frac{(\partial_k v(x)- \partial_kv(y)) {(\partial_j v(x)- \partial_j v(y))}}{\abs{x-y}^{2\beta_0}}\ \frac{\mathrm{d} x\ \mathrm{d} y}{\abs{x-y}^d}\\ 
		&\quad+\int_{\R^d} \int_{\R^d} \frac{(a_{jk}(t,x)- a_{jk}(t,y)) \partial_k v(y) {(\partial_j v(x)- \partial_j v(y))}}{\abs{x-y}^{2\beta_0}}\ \frac{\mathrm{d} x\ \mathrm{d} y}{\abs{x-y}^d} \\
		&\quad+\int_{\R^d} a_{jk}(t,x) \partial_k v(x) {\partial_j v(x)} \ \mathrm{d}{x} \biggr]\\
		&\ge   \sum_{k=1}^d \frac\eta2 \frb(\partial_k v, \partial_k v) \\
		&\quad- \sum_{j,k=1}^d \frac 1 {2\eta} \int_{\R^d} \int_{\R^d} \frac{\abs{a_{jk}(t,x)- a_{jk}(t,y)}^2 \abs{\partial_k v(y)}^2}{\abs{x-y}^{2\beta_0}}\ \frac{\mathrm{d} x\ \mathrm{d} y}{\abs{x-y}^d},
	\end{align*}
	where we used the ellipticity of $(a_{jk})$, Hölder's inequality and Young's inequality ($a\cdot b \le \frac\eta 2 a^2+\frac 1 {2\eta}b^2,\ a,b\in\R$). 
	It remains to estimate the last term by $\frac \eta 4 \sum_{j,k=1}^d \frb(\partial_k v, \partial_k v)+ C \norm{v}_{H^{\beta_0}}^2$ for some $C\ge0$.
	This can be done as in the proof of Lemma~\ref{lem:comest} using the regularity of $(a_{jk})$. Thus $\fra$ is quasi-coercive.
	
	Let $f \in L^2(I; L^2(\R^d))$ and $u_0 \in H^1(\R^d)$.
	By Theorem~\ref{thm:MRinRforformsonI} there exists a unique $u \in \textit{MR}_{\alpha_0}(\A)$ such that
	\(
		u'+\A u =f
	\) in $L^2(I;(H^{1+\beta_0}(\R^d))')$ and $u(0)=u_0$.
	Since the identity 
	\begin{equation*}
		\langle v, w \rangle_{(H^{1+\beta_0}(\R^d))', H^{1+\beta_0}(\R^d)}  = (v \, \vert \, B w)_{L^2(\R^d)} \quad (v \in L^2(\R^d),\, w \in H^{2\beta_0}(\R^d))
	\end{equation*}
	holds and since $\partial_j B v = B \partial_j v$ for every $v \in H^{1+2\beta_0}(\R^d)$, $j\in\{1,\dots, d\}$
	we see that the function $u$ is the desired solution.
\end{proof}


\subsection*{Parabolic systems}
Let $I=[0,T]$ with $T>0$, $d, n \in \N$.
Let $a_{jk}^{lm} \in L^\infty(I; BUC(\R^d))$ for $j,k \in \{1, \dots, d\}$, $l, m \in \{1, \dots, n\}$ and suppose that there exists $\eta >0$ such that
\begin{equation*}
	\sum_{j,k=1}^d \sum_{l,m=1}^n a_{jk}^{lm}(t,x) \zeta_l \zeta_m \xi_j \xi_k \ge \eta \abs{\zeta}^2 \abs{\xi}^2 \quad (t\in I,\ x \in \R^d,\ \zeta \in \R^n,\ \xi \in\R^d).
\end{equation*}
Note that this condition is called the Legendre-Hadamard ellipticity condition.

\begin{corollary}
	Let $\alpha \in [0,1/2]$ and if $\alpha>0$ suppose in addition that $a_{jk} \in \mathring W^{\alpha+\delta,\frac1{\alpha}}(I; L^\infty(\R^d))$ for some $\delta>0$.
	Then for every $f \in L^2(0,T;H^{2\alpha -1}(\R^d))^n$, $u_0 \in H^{2\alpha}(\R^d)^n$ there exists a unique $u \in H^\alpha(I;H^1(\R^d))^n \cap H^1(I;H^{2\alpha-1}(\R^d))^n$ such that
	\begin{equation*}
		u_l'  -  \sum_{m=1}^n \operatorname{div} \big((a_{jk}^{lm})_{j,k \in \{1,\dots,d\}} \nabla u_m \big) =  f_l \quad (l \in \{1,\dots, n\}),\quad
                       u(0) =u_0.
	\end{equation*}
\end{corollary}
Note that the domain of the elliptic operator is time dependent.
\begin{proof}
	We define the non-autonomous form $\fra \colon I \times H^1(\R^d)^n \times H^1(\R^d)^n\to \R$ by 
	\begin{equation*}
		\fra(t,v,w) = \int_{\R^d}\sum_{j,k=1}^d \sum_{l,m=1}^n  a_{jk}^{lm}(t,x) \partial_j v_m \, {\partial_k w_l} \ \mathrm{d}{x}. 
	\end{equation*}
	Then $\fra$ satisfies the conditions of Theorem~\ref{thm:MRinRforformsonI}.
	
	Indeed, boundedness and time regularity is a direct consequence of the assumptions above. Furthermore, quasi-coercivity may be obtained by localization of the coefficients $a_{jk}^{lm}(t,\cdot)$ and Plancherel's theorem (see \cite[Theorem~3.25]{GM05}).
\end{proof}


\subsection*{Time dependent generalized fractional Laplacians}
Let $I=[0,T]$ with $T>0$, $d \in \N$, $\alpha \in [0, 1/2]$ and $\beta \in (0,1)$.
Suppose $K \colon I \times \R^d\times \R^d \to \R$ is measurable, $K(t,\cdot,\cdot)$ is symmetric for all $t\in I$ and there exist constants $0<\eta<M$ such that $\eta\le K(t,x,y) \le M$ for all $t\in I$, $x,y \in \R^d$.
\begin{corollary}
	Let $\alpha \in [0, 1/2]$ and $\beta \in (0,1)$. In addition, suppose that $K \in \mathring W^{\alpha+\delta, \frac 1 \alpha}(I; L^\infty(\R^{2d}))$ for some $\delta>0$.
	Then for every $f \in L^2(I ;H^{(2\alpha -1)\beta}(\R^d))$, $u_0 \in H_0^{2\alpha \beta}(\R^d)$ there exists a unique $u \in H^\alpha(I;H^\beta(\R^d)) \cap H^1(I;H^{(2\alpha-1)\beta}(\R^d))$ such that
	\begin{equation*}
          	u'  +  \operatorname{p.v.} \int_{\R^d} K(t,x,y) \frac{(u(t,x)-u(t,y))}{\abs{x-y}^{2\beta+d}}\ \mathrm{d} y =  f, \quad u(0) =u_0.
	\end{equation*}
\end{corollary}
\begin{proof}
	We define the bounded and quasi-coercive non-autonomous form $\fra \colon I \times H^{\beta}(\R^d) \times H^{\beta}(\R^d) \to \R$ by
\[
	\fra(t,v,w) = \int_{\R^d} \int_{\R^d} K(t,x,y)\frac{(v(x)-v(y))(w(x)-w(y))}{\abs{x-y}^{2\beta}}\ \frac{\mathrm{d} x\ \mathrm{d} y}{\abs{x-y}^d}
\]
and we denote by $\A$ its associated operator. 
It is easy to check that $\A(\cdot) \in \mathring W^{\alpha+\delta, \frac 1 \alpha}(I;\mathcal{L}(H^{\beta}(\R^d), (H^{\beta}(\R^d))'))$.
Thus we may apply Theorem~\ref{thm:MRinRforformsonI}.
For the case $\alpha = \frac 1 2$, note that $D(A(0)^{1/2})=H^{\beta}(\R^d)$, where $A(0)$ denotes the part of $\A(0)$ in $H$, since $\fra(t,\cdot, \cdot)$ is symmetric.
\end{proof}

\section{Appendix: Vector valued fractional calculus}
The material covered in this appendix is well known, despite that we could not find all the needed results in the literature.
More results about vector valued Sobolev spaces can be found in \cite{Ama97}, \cite{MS12} and \cite{Sim90}.

Let $\alpha \in (0,1)$, let $p \in[1,\infty)$, let $I \subset \R$ be an interval and let $E$ be a Banach space.
Given a measurable function $f \colon I \to E$ we set
\begin{equation*}
	[f]_{W^{\alpha,p}(I;E)} := \left(\int_I\int_I \left(\frac{\norm{f(t)-f(s)}_{E}}{\abs{t-s}^{\alpha}} \right)^p \ \frac{ \mathrm{d}{t}\ \mathrm{d}{s}}{\abs{t-s}}\right)^{1/p}.
\end{equation*}
We define the homogeneous fractional Sobolev space
\[
	\mathring W^{\alpha,p}(I;E) := \{ f\colon I \to E \text{ measurable} : [f]_{W^{\alpha,p}(I;E)} < \infty\}
\]
and the fractional Sobolev space
\[
	 W^{\alpha,p}(I;E) := \{ f \in L^p(I; E) : [f]_{W^{\alpha,p}(I;E)} < \infty\}.
\]
Note that $[\cdot]_{W^{\alpha,p}(I;E)}$ is a seminorm
and $\bigl(W^{\alpha,p}(I;E) , \norm{\cdot}_{W^{\alpha,p}(I;E)} \bigr)$ is a Banach space where
\[
	\norm{f}_{W^{\alpha,p}(I;E)} := \left( \norm f_{L^{p}(I;E)}^p + [f]_{W^{\alpha,p}(I;E)}^p \right)^{1/p}.
\]
We collect some well known results about these spaces.
\begin{proposition}\label{prop:extension}
	Let $\alpha \in (0,1)$, let $p \in[1,\infty)$, let $I =(a,b)$, where $-\infty <a<b<\infty$ and let $E$ be a Banach space.
	There exists an operator $\mathcal E \in \mathcal{L}(W^{\alpha,p}(I;E), W^{\alpha,p}(\R;E))$, such that
	$(\mathcal Ef)\vert_I = f$ for all $f\in W^{\alpha,p}(I;E)$ and $\operatorname{supp} (\mathcal Ef) \subset (a-(b-a), b+ (b-a))$.
\end{proposition}
\begin{proof}
	Let $\varphi \in \D(\R)$ with $\varphi(t)=1$ for all $t \in I$ and $\operatorname{supp} \varphi \subset (a-(b-a), b+ (b-a))$.
	It is easy to check that the mapping defined by
	\[
		(\mathcal Ef)(t) :=
		\begin{cases}
			f(t), & t\in I\\
			\varphi(t)f(2a-t), & t\in(a-(b-a),a)\\
			\varphi(t)f(2b-t), & t\in(b,b+(b-a))\\
			0, &\text{else}
		\end{cases}
	\]
	has the desired properties.
\end{proof}

\begin{proposition}\label{prop:dense}
	Let $\alpha \in (0,1)$, let $p \in[1,\infty)$, let $I \subset \R$ be an interval and let $E$ be a Banach space.
	The $E$ valued test functions $\D(\R;E)$ are dense in $W^{\alpha,p}(\R;E)$
	and the smooth functions $C^\infty(I;E)$ are dense in $W^{\alpha,p}(I;E)$.
\end{proposition}
\begin{proof}
	Since $\D(\R;E)$ is dense in $W^{1,p}(\R;E)$ and since
	\[
		W^{\alpha,p}(\R;E) = \left( L^p(\R;E) , W^{1,p}(\R;E) \right)_{\alpha, p}
	\]
	we obtain that $\D(\R;E)$ is also dense in $W^{\alpha,p}(\R;E)$ by \cite[p.~39]{Tri95}.
	
	The second statement follows by Proposition~\ref{prop:extension} and the first statement.
\end{proof}

Let $\alpha \in (0,1)$, let $p \in[1,\infty)$, let $I \subset \R$ be an interval and let $E$ be a Banach space.
Given a function $f \colon I \to E$ we set
\begin{equation*}
	[f]_{C^{\alpha}(I;E)} := \sup_{s,t \in I} \frac{\norm{f(t)-f(s)}_E}{\abs{t-s}^\alpha}.
\end{equation*}
We define the space of Hölder continuous functions
\[
	C^{\alpha}(I;E) := \{ f\in C(I;E) : [f]_{C^{\alpha}(I;E)} < \infty\}.
\]
Note that $C^{\alpha}(I;E)$ is a Banach space with respect to the norm
\[
	\norm{f}_{C^{\alpha}(I;E)} := \norm{f}_{L^\infty(I;E)} + [f]_{C^{\alpha}(I;E)}.
\]

By \cite[Corollary~26]{Sim90} we have what follows.
\begin{proposition}\label{prop:Hölder}
	Let $\alpha \in (0,1)$, let $p > \frac 1 \alpha$, let $I \subset \R$ be an interval and let $E$ be a Banach space.
	Then
	\begin{equation*}
		\mathring W^{\alpha,p}(I;E) \hookrightarrow \mathring C^{\alpha-\frac 1 p}(I;E).
	\end{equation*}
\end{proposition}

If $\alpha > \frac 1 p$ we identify the function $f \in W^{\alpha,p}(I;E)$ with its continuous version.

\begin{proposition}\label{prop:const_ext}
	Let $\alpha \in (0,1)$, let $p > \frac 1 \alpha$, let $I \subset \R$ be an interval and let $E$ be a Banach space.
	For any $t \in \overline I$ the mappings
	\[
		\mathcal E_t^r \colon \mathring W^{\alpha,p}(I;E) \to \mathring W^{\alpha,p}(I \cup [t,\infty);E), \quad (\mathcal E_t^r f)(s) = \begin{cases} u(s), & s\in I,\, s<t\\ u(t), & s \ge t \end{cases}
	\]
	and
	\[
		\mathcal E_t^l \colon \mathring W^{\alpha,p}(I;E) \to \mathring W^{\alpha,p}(I \cup (-\infty,t];E), \quad (\mathcal E_t^l f)(s) = \begin{cases} u(s), & s\in I,\, s>t\\ u(t), & s \le t \end{cases}
	\]
	define bounded operators.
\end{proposition}
\begin{proof}
	We consider the interval $I=[0,T]$ and the operator $\mathcal E_0^l$, the other cases are similar.
	Let $f \in C^\infty(I;E)$. By subtracting $f(0)$ we may assume $f(0)=0$.
	We have
	\begin{multline*}
		[\mathcal E_0^l f]_{W^{\alpha,p}((-\infty,T];E)}^p\\ = \int_I\int_I \left(\frac{\norm{f(t)-f(s)}_{E}}{\abs{t-s}^{\alpha}} \right)^p \ \frac{ \mathrm{d}{t}\ \mathrm{d}{s}}{\abs{t-s}} 
			+ 2\int_I\int_{(-\infty,0)} \left(\frac{\norm{f(s)}_{E}}{(s-t)^{\alpha}} \right)^p \ \frac{ \mathrm{d}{t}\ \mathrm{d}{s}}{s-t}\\
			=[f]_{W^{\alpha,p}(I;E)}^p + \frac 2 {\alpha p} \int_I \left(\frac{\norm{f(s)}_{E}}{s^{\alpha}} \right)^p \ \mathrm{d}{s}.
	\end{multline*}
	Since $f$ is smooth and thus also Lipschitz continuous, the second term on the right hand side is finite.
	By \cite[p.~745, (6.8)]{PSS07} we have that 
	\[
		\left(\int_I \left(\frac{\norm{f(s)}_{E}}{s^{\alpha}} \right)^p \ \mathrm{d}{s} \right)^{1/p} \le \frac{1+\alpha- 1/ p}{\alpha- 1/ p} [f]_{W^{\alpha,p}(I;E)}.
	\]
	Hence $[\mathcal E_0^l f]_{W^{\alpha,p}((-\infty,T];E)}^p \le C [f]_{W^{\alpha,p}(I;E)}$ for some constant $C$ depending only on $\alpha$ and $p$.
	Finally, by Proposition~\ref{prop:dense} and Proposition~\ref{prop:Hölder} this estimate extends to arbitrary $f \in W^{\alpha,p}(I;E)$.
\end{proof}

Let $H$ be a Hilbert space and let $\alpha\in [0,\infty)$.
We define 
\[
	H^\alpha(\R;H):= \{ u \in L^2(\R;H) : \abs{\xi}^\alpha \hat u(\xi) \in L^2(\R;H) \}
\]
and for $u \in H^\alpha(\R;H)$ we define $\partial^\alpha u$ by $(\partial^\alpha u)\hat{}\,(\xi) = (i\xi)^\alpha\hat u (\xi)$
and $\abs{\partial}^\alpha u$ by $(\abs{\partial}^\alpha u)\hat{}\,(\xi) = \abs{\xi}^\alpha\hat u (\xi)$.
Then $H^\alpha(\R;H)$ is a Hilbert space with the norm
\[
	\norm{u}_{H^\alpha_H}^2 := \norm{u}_{L^2_H}^2 + \norm{\partial^\alpha u}_{L^2_H}^2.
\]
\begin{proposition}\label{prop:isometry}
	For $\alpha \in (0,1)$ we have $H^\alpha(\R;H)=W^{\alpha,2}(\R;H)$ with
	\[
		C_\alpha \norm{\partial^\alpha u}_{L^2_H}^2 = [f]_{W^{\alpha,2}(\R;H)}^2,
	\]
	where $C_\alpha := 2\int_\R \frac{1-\cos{s}}{\abs{s}^{1+2\alpha}} \ \mathrm{d}{s}$.
\end{proposition}
\begin{proof}
	First note that
	\begin{equation*}
		C_\alpha \abs{\xi}^{2\alpha}  = 2\int_\R \frac{1-\cos(\xi h)}{\abs{h}^{1+2\alpha}} \ \mathrm{d}{h} = \int_\R \frac{\abs{1-e^{i\xi h}}^2}{\abs{h}^{1+2\alpha}} \ \mathrm{d}{h}
	\end{equation*}
	by substituting $s$ with  $h\abs{\xi}$.
	Let $u \in L^2(\R;H)$.
	By Plancherel's theorem and by Fubini's theorem we have
	\begin{multline*}
		C_\alpha \norm{\partial^\alpha u}_{L^2_H}^2 =  C_\alpha\norm{\abs{\xi}^\alpha \hat u(\xi)}_{L^2_H}^2
		= \int_\R C_\alpha \abs{\xi}^{2\alpha} \norm{ \hat u(\xi)}_{H}^2 \ \mathrm{d}{\xi} \\
		= \int_\R \int_\R \frac{\abs{1-e^{i\xi h}}^2}{\abs{h}^{1+2\alpha}} \ \mathrm{d}{h} \norm{ \hat u(\xi)}_{H}^2 \ \mathrm{d}{\xi}
		= \int_\R \int_\R \frac{\norm{(1-e^{i\xi h}) \hat u(\xi)}_{H}^2}{\abs{h}^{1+2\alpha}}   \ \mathrm{d}{\xi} \ \mathrm{d}{h}\\
		= \int_\R \int_\R \frac{\norm{u(t)-u(t+h)}_{H}^2}{\abs{h}^{1+2\alpha}}   \ \mathrm{d}{t} \ \mathrm{d}{h}
		= [f]_{W^{\alpha,2}(\R;H)}^2. \tag*\qedhere
	\end{multline*}
\end{proof}
By polarization we also have 
\[
	C_\alpha(\partial^\alpha u \, \vert \, \partial^\alpha v)_{L^2_H} = \int_\R \int_\R \frac{(u(t)-u(s) \, \vert \, v(t)-v(s))_{H}}{\abs{t-s}^{1+2\alpha}}   \ \mathrm{d}{t} \ \mathrm{d}{s} \quad (u,v \in H^\alpha(\R;H)).
\]

\begin{corollary}\label{cor:isometry}
	For every $u \in H^\alpha(\R;V')$ and every $v \in H^\alpha(\R;V)$ we have
	\[
		C_\alpha \int_\R \langle\partial^\alpha u , \partial^\alpha v\rangle \ \mathrm{d}t = \int_\R \int_\R \frac{\langle u(t)-u(s), v(t)-v(s)\rangle}{\abs{t-s}^{1+2\alpha}}   \ \mathrm{d}{t} \ \mathrm{d}{s}.
	\]
\end{corollary}

\begin{proposition}\label{prop:weakderivative}
	Let $u \in L^2(\R;H)$ and $\alpha > 0$. Then $u \in H^\alpha(\R;H)$ if and only if there exists a $v \in L^2(\R;H)$ such that
	\[
		(u \, \vert \, \partial^{\alpha*} \varphi)_{L^2_H} = (v \, \vert \, \varphi)_{L^2_H} \quad (\varphi \in \D(\R;H)).
	\]
\end{proposition}
\begin{proof}
	By Plancherel's theorem we obtain
	\[
		( \hat u \, \vert \, \overline{(i\xi)^\alpha} \hat \varphi)_{L^2_H} = ( \hat v \, \vert \, \hat \varphi)_{L^2_H} \quad (\varphi \in \D(\R;H)).
	\]
	Hence, by the density of $\D(\R;H)$ in $L^2(\R;H)$ it follows that $(i\xi)^\alpha \hat u = \hat v$.
\end{proof}

\begin{lemma}\label{lem:compembedding}
	Let $\alpha \in (0,\frac 1 2]$ and let $2<p< \frac 2{1-2\alpha}$ if $\alpha < \frac 1 2$ and $2<p< \infty$ if $\alpha = \frac 1 2$. 
	Then $H^{\alpha}(\R;H) \hookrightarrow L^p(\R;H)$ and for every $\varepsilon >0$ there exists $c_\varepsilon$ such that
	\begin{equation*}
		\norm{v}_{L^p(\R;H)} \le \varepsilon \norm{\partial^{\alpha} v}_{L^2(\R;H)} + c_\varepsilon\norm{v}_{L^2(\R;H)}
	\end{equation*}
	for all $v \in H^{\alpha}(\R;H)$.
\end{lemma}
\begin{proof}
	Let $p'$ be such that $1=\frac 1 p + \frac 1 {p'}$ and $q$ such that $\frac1{p'}=\frac 1 2+ \frac 1 q$. 
	Note that $p' = \frac p{p-1}$ and $q=\frac{2}{1-2/p} > \frac 1 \alpha$.
	For $\rho >0$ and $v \in H^{\alpha}(\R;H)\cap L^p(\R;H)$ we have
	\begin{multline*}
		\norm{v}_{L^p(\R;H)} = \norm{\hat {\hat{v}}}_{L^p(\R;H)} \le c_p \norm{\hat{v}}_{L^{p'}(\R;H)} \\
		\le \norm{(\rho^{1/\alpha} + \abs\cdot)^{-\alpha}}_{L^{q}(\R)}   \norm{(\rho^{1/\alpha} + \abs\cdot)^{\alpha}\hat{v}}_{L^{2}(\R;H)}\\
		\le \left(\frac{2}{q\alpha -1}\right)^{1/q} \rho^{\frac{1}{q\alpha}-1} \left(\rho\norm{v}_{L^{2}(\R;H)}+ \norm{\partial^{\alpha} v}_{L^2(\R;H)} \right).
	\end{multline*}
	This estimate proves the claim for $\rho$ sufficiently large.
\end{proof}

Let $I \subset \R$ be an interval.
We define $H^\alpha(I;H) := \{f\vert_I : f \in H^\alpha(\R;H) \}$ with
\[
	\norm{f}_{H^\alpha(I;H)} := \inf  \{ \norm{g}_{H^\alpha(\R; H)} : g \in H^\alpha(\R;H),\, g\vert_I =f \}.
\]
Since $H^\alpha(I;H)$ is isometric isomorphic to the quotient space $H^\alpha(\R; H) / \{f \in H^\alpha(\R; H) : f \vert_I = 0 \}$ it is also a Hilbert space.
Furthermore, by Proposition~\ref{prop:extension} and Proposition~\ref{prop:isometry} we have $H^\alpha(I;H) = W^{\alpha,2}(I;H)$ with equivalent norms for $\alpha \in (0,1)$.

\def\cprime{$'$}
\providecommand{\bysame}{\leavevmode\hbox to3em{\hrulefill}\thinspace}

\noindent
\emph{Dominik Dier}, \emph{Rico Zacher}, Institute of Applied Analysis, 
University of Ulm, 89069 Ulm, Germany,
\texttt{dominik.dier@uni-ulm.de}, \texttt{rico.zacher@uni-ulm.de}

\end{document}